\documentclass[11pt,letterpaper]{article}  
\usepackage{epsfig}
\usepackage{amsfonts}
\usepackage{amssymb}
\usepackage{amsmath}
\usepackage{float}

\usepackage{color}
\usepackage{empheq}
\usepackage{graphicx,colordvi}
\newcommand{\ba}{\begin{array}}
\newcommand{\ea}{\end{array}}
\newcommand{\bal}{\begin{align}}
\newcommand{\eal}{\end{align}}
\newcommand{\be}{\begin{equation}}
\newcommand{\beqn}{\begin{equation}}
\newcommand{\eeqn}{\end{equation}}
\newcommand{\bea}{\begin{eqnarray}}
\newcommand{\eea}{\end{eqnarray}}
\newcommand{\benum}{\begin{enumerate}}
\newcommand{\eenum}{\end{enumerate}}
\newcommand{\bi}{\begin{itemize}}
\newcommand{\ei}{\end{itemize}}
\newcommand{\bean}{\begin{eqnarray*}} 
\newcommand{\eean}{\end{eqnarray*}}   

\def\diag{\mbox{\rm Diag}}

\def\I{\rm I\!\!I}

\def\les{\leqslant}
\def\ges{\geqslant}

\def\les{\leqslant}
\def\ges{\geqslant}

\def\ms{\medskip}

\def\bde{\begin{definition}\label}
\def\ede{\end{definition}}
\def\be{\begin{equation}}
\def\bel{\begin{equation}\label}
\def\ee{\end{equation}}
\def\bt{\begin{theorem}\label}
\def\et{\end{theorem}}
\def\bc{\begin{corollary}\label}
\def\ec{\end{corollary}}
\def\bl{\begin{lemma}\label}
\def\el{\end{lemma}}
\def\bp{\begin{proposition}\label}
\def\ep{\end{proposition}}
\def\bas{\begin{assumption}\label}
\def\eas{\end{assumption}}
\def\br{\begin{remark}\label}
\def\er{\end{remark}}
\def\bex{\begin{example}\label}
\def\ex{\end{example}}
\def\ba{\begin{array}}
\def\ea{\end{array}}
\newtheorem{theorem}{Theorem}
\newtheorem{definition}[theorem]{Definition}
\newtheorem{lemma}[theorem]{Lemma}
\newtheorem{corollary}[theorem]{Corollary}
\newtheorem{proposition}[theorem]{Proposition}
\newtheorem{example}[theorem]{Example}
\newtheorem{remark}{Remark}

\allowdisplaybreaks

\begin{document}

\title{\LARGE\bf Social Optima in Mean Field Linear-Quadratic-Gaussian Control with Volatility Uncertainty
}


\author{Jianhui Huang\thanks{J.~Huang is with the Department of Applied Mathematics,
The Hong Kong Polytechnic University, Hong Kong
({majhuang@polyu.edu.hk}).},\quad
Bing-Chang Wang\thanks{B.-C.~Wang is with the School of Control Science and Engineering, Shandong University, Jinan 250061, P. R. China.
  ({bcwang@sdu.edu.cn}).},\quad
and Jiongmin Yong\thanks{J.~Yong is with the Department of Mathematics
University of Central Florida Orlando, FL 32816 ({jiongmin.yong@ucf.edu}). This author was supported in part by NSF DMS-1812921.}
  }


\maketitle
%

\begin{abstract}
This paper examines mean field linear-quadratic-Gaussian (LQG) social optimum control with volatility-uncertain common noise. The diffusion terms in the dynamics of agents contain an unknown volatility process driven by a common noise. We apply a robust optimization approach in which all agents view volatility uncertainty as an adversarial player. Based on the principle of person-by-person optimality and a \emph{two-step-duality} technique for stochastic variational analysis, we construct an auxiliary optimal control problem for a representative agent. Through solving this problem combined with a consistent mean field approximation, we design a set of decentralized strategies, which are further shown to be asymptotically social optimal by perturbation analysis.
 \end{abstract}


\section{Introduction}

\subsection{Social optimum control by mean-field analysis}
The large population (LP) systems have been found wide applications across a broad spectrum including: economics, biology, engineering, and social science \cite{BFY13, C14, GS13, weintraub2008markov}. The most significant character of LP system, is the interactive weakly-coupling structure across considerable agents: each individual influence on whole system is negligible, but their overall population impact is substantial and cannot be ignored. Recently, dynamic decisions of LP system become much more important along with recent rapid growth of practical decision system exhibiting large-scaled interactive behaviors. The mean field game (MFG) has drawn intensive research attention because it provides an effective theoretical scheme to analyze asymptotic behavior of \emph{controlled} LP systems. Note that LP system studied by MFG involves numerous \emph{competitive} players (or agents) endowed with competitive cost functionals.  

MFG has been extensively studied from various viewpoints on either linear-quadratic-Gaussian (LQG) setting \cite{HCM07, LZ08, WZ13, MB17} or general nonlinear one \cite{HMC06, LL07, CD13}. In addition, MFG with common noise is motivated by a variety of scenarios in finance and economics such as \emph{system risk} \cite{CDL14, GLL11}.

Apart from noncooperative MFG, social optimum control by mean field analysis has also drawn increasing attention recently. The social optimum problem refers to a LP system in which all players cooperate to optimize some common social cost---the sum of individual costs. Social optima are linked to a type of team decision \cite{H80} but with highly complex interactions. All agents in team decision access different information sets, thus social optima are \emph{decentralized} and differ from classical vector optimization with centralized designer. When player number $N\longrightarrow +\infty,$ some mean-field team-optimization problem is inspired to study the asymptotic behavior of LP system with two approaches along this line: the direct method \cite{LL07, HZ19} and fixed-point method. We list few relevant works for the second one. Huang \emph{et al.} consider social optima in mean field LQG control, and provide an asymptotic team-optimal solution \cite{HCM12}.
Wang and Zhang investigate a mean field social optimal problem in which a Markov jump parameter appears as a common source of randomness for all agents \cite{WZ16}. 
For further literature, see \cite{AM15} for team-optimal control with finite population and partial information,
and \cite{SNM18} for the dynamic 
collective choice by finding a social optimum.

\subsection{Volatility uncertainty with common noise}
Motivated by the aforementioned studies, the present study explores a class of robust cooperative mean field social optimum problems. Specifically, we focus on team optimization in an LQG setup with symmetric minor agents, driven by common noise but with uncertainty in its volatility term. More details of the motivation behind our problem are presented as follows.


\ms

In \cite{HH16}, the authors investigate mean field models with a \emph{global} uncertainty term, which means that all players share a common unknown deterministic disturbance. 
They adopted the ``soft constraint" approach \cite{BB95,E06} by removing the bound of disturbance while the effort is simultaneously penalized in cost function. The studies \cite{TBT13, MB17} consider the case where each agent is paired with its \emph{local} disturbance, and provide an $\epsilon$-Nash equilibrium by tackling a Hamilton-Jacobi-Isaacs equation combined with fixed-point analysis. Another study relevant to our work is \cite{WH17} which presents robust analysis of mean-field social control with uncertain drift only. Because of the absence of volatility uncertainty therein, a closed-loop strategy with a consistency condition is still admissible in terms of a standard Riccati equation. In addition, asymptotic social optimality could still be verified in \cite{WH17} directly based on a stationary condition of the strategy specified by Riccati equations obtained.

\ms

Unlike \cite{WH17}, this paper is devoted to volatility uncertainty of social optimum control in mean field LQG setup with common noise. Notice that various studies of mathematical finance (e.g., pricing and hedging \cite{ALP95, L95, MN18}) have remarkably focused on markets with uncertain volatility.
In \cite{Buff02}, uncertain volatility models are introduced to evaluate a scenario where the volatility coefficient of the pricing model cannot be determined exactly. Therefore, a practical motivation here is, in many decision problems, a large number of coupled decision markers share a common noise but with uncertain volatility on it. For instance, volatility of trading prices in a financial market is often unknown and the \emph{implied volatility} has thus been inspired and well-studied. Subsequently, when some cooperative investors concern their team optima, 
it becomes necessary to study the social optimization with volatility uncertainty. Another example is system risk minimization in an inter-banking system: all branches (of team formation) are subject to some uncertainty in common system noise thus robust volatility analysis arises when seeking optima in joint operations.  So, it is worthwhile to study cooperative mean field model with volatility uncertainty \cite{BB95, Buff02, LPVD83}. Moreover, for \emph{linear} dynamics (e.g., wealth process in Merton model), their volatilities are often inexact by allowing some modeling errors; thus, when some \emph{quadratic hedging} is considered, the LQG setup is suggested and we adopt it here.

\subsection{The analysis outline and comparison}

At the first glance, this present study seems somewhat similar to our previous study \cite{WH17} and \cite{HH16}. However, various subtle and essential differences exist between them, in both setup and analysis. We highlight some key differences below for more clear comparison.

\ms

(i) Our present study examines the uncertainty of team optimization; thus, a variational analysis should be conducted to test the response of related componentwise Fr\'{e}chet differentials for a given agent.
Such an analysis is not required in \cite{HH16} when studying the uncertainty of mean field games when all agents are competitive.

\ms

(ii) In team optimization, a key step is to verify (uniform) convexity of social cost functional, that is high-dimensional. For team optimization (e.g., \cite{HCM12}) with standard assumption (SA), such convexity follows directly because the SA weights are all positive (nonnegative) definite. However, it becomes more challenging in present study because some weights are intrinsically indefinite due to soft constraint and min-max setup here. Even though negative weight is also addressed in \cite{HH16} but (uniform) convexity therein is more tractable: only low-dimensional optimization needs to be treated in \emph{competitive} game context. More precisely, in \cite{HH16}, we need only to consider perturbation for a given single agent to verify the approximate Nash equilibrium by fixing other agents' strategies. However, the present study must consider team perturbation for all agents instead single one only; thus, the convexity involved is high-dimensional and indefinite, which becomes more technical to be checked.

\ms

(iii) Uncertainty (disturbance) in \cite{HH16} is postulated to be \emph{deterministic} on the \emph{drift} term only. Thus, the related consistency condition by the fixed-point argument, reduces to a forward-backward ordinary differential equation (FBODE), for which the well-posdness is more tractable. For instance, \emph{the compatibility method} in \cite{MY99} still works in \cite{HH16} to such FBODE, but fails here to the more complicated forward-backward stochastic differential equation (FBSDE) of consistency condition due to volatility uncertainty.

\ms

(iv) Unlike \cite{WH17} for team optimization with drift uncertainty only, volatility uncertainty imposed here brings more technical difficulties. For example, more subtle estimates for a fully-coupled consistency FBSDE system, especially for its (backward) adjoint solution in common noise component, should be invoked. Moreover, for the auxiliary problem construction for social optimality, the related variational analysis becomes rather involved (see Section 5); furthermore, it differs fairly to that of \cite{WH17} mainly because of common noise and volatility uncertainty. More crucially, a \emph{two-step duality} procedure (see Section 4.2) should be applied and a new type of auxiliary problem is constructed; whereas in \cite{WH17}, only single-step duality is required. In addition, different from \cite{WH17}, the consistency system here requires a new \emph{embedding representation} type.

\ms

To conclude, the main contributions of this paper can be summarized as follows:

\ms

(1) The volatility uncertainty of team optimization on common noise is introduced and formulated in a soft-constraint setting. Two sequential optimization problems are also formulated.

\ms

(2) An auxiliary control problem is constructed via a two-step duality procedure, and the consistency system is obtained through embedding representation of a non-standard mean-field type FBSDE. The related uniform convexity (concavity) is also established in high dimension case.

\ms

(3) We obtain global solvability of related 
FBSDEs in some nontrivial and nonstandard case.

\ms

(4) The decentralized optimal team strategy is derived in an open-loop sense, and its asymptotic social optimality is verified in a robust social sense.

\ms

The rest of this paper is organized as follows. Section 2 formulates the volatility uncertainty with soft constraint; Section 3 discusses the control problem with volatility uncertainty; Section 4 investigates team optimization in person-by-person optimality; based on this, Section 5 designs the decentralized strategies through a consistency condition system; Section 6 analyzes the well-posedness of FBSDEs, which arises from the consistency system; Section 7 presents asymptotic robust social optimality of the decentralized strategy; Section 8 concludes the paper.

%
%
%

%
%
%
%
%
%

\section{Problem formulation}\label{sec2}

We denote by $\mathbb{R}^{k}$ the $k$-dimensional Euclidean space, $\mathbb{R}^{n\times k}$ the set of all $n\times k$ matrices,
{$\otimes$ the Kronecker product}. We use $|\cdot|$ to denote the norm of an Euclidean space, or the Frobenius norm of matrices. For a vector or matrix $M$, $M^T$ denotes its transpose; for two vectors $x,y$, $\langle x,y\rangle=x^Ty$. For symmetric matrix $Q$ and a vector $z$, $|z|_Q^2= z^TQz$, and $Q>0$ ($Q\ges0$) means that $Q$ is positive (nonnegative) definite. 
Consider a finite time horizon $[0,T]$ for $T>0$, for a given filtration $\mathbb{G}\triangleq\{\mathcal{G}_{t}\}_{0\les t\les T}$, denote $L^{2}_{\mathbb{G}}(0, T; \mathbb{R}^\ell)$ ($L_{\mathbb{G}}^2(\Omega;C([0,T];\mathbb{R}^\ell))$) the space of all $\mathbb{R}^\ell$-valued $\mathcal{G}_t$-progressively measurable (continuous) processes $s(\cdot)$ satisfying $\|s\|_{L_2}^2:=\mathbb{E}\int_0^{T}|s(t)|^{2}dt<\infty$ ($\|s\|_{\max}^2:=\mathbb{E}\sup_{0 \les t \les T}|s(t)|^{2}<\infty$). For convenience of presentation, we may use $c$ (or $ c_1, c_2, \ldots$) to denote a generic constant which does not depend on the population size $N$ of LP system, and may vary from place to place.


\ms

Let $(\Omega,
\mathcal F, \mathbb{F}, \mathbb{P})$ be a complete
filtered probability space on which a sequence of independent one-dimensional Brownian motions $\{W_i(t), i=0,1,\cdots, N\}$ are defined, where $\mathbb{F}=\{\mathcal F_t\}_{0 \les t \les T}$ is the natural filtration of $\{W_i(t), i=0,1,\cdots, N\}$ augmented by all the $\mathbb{P}$-null sets in $\mathcal F.$ Consider a linear stochastic LP systems with $N$ agents (or, particles), in which the $i^{th}$-agent $\mathcal{A}_{i}$ evolves by
\begin{eqnarray}\label{eq1}
dx_i(t) &\hspace*{-0.1cm}=\hspace*{-0.1cm}& [Ax_i(t)+Bu_i(t)+ f(t)]dt+[Du_i(t)+\sigma(t)]dW_i(t)
\cr&&+[C_0x_i(t)+D_0u_i(t)+\sigma_0(t)]dW_0(t), \quad x_i(0)=x_0, \quad  i=1,\cdots, N,
\end{eqnarray}
where $x_i(\cdot)$ and $u_i(\cdot)$ are state and input of agent $\mathcal{A}_{i}$, valued in $\mathbb{R}^n$ and $\mathbb{R}^r$, respectively, and $x_0\in \mathbb{R}^n$ is a constant vector; Coefficients $A,B,D,C_0,D_0$ are constant matrices of suitable sizes;
$W_i(\cdot)$ is a Brownian motion representing the idiosyncratic noise for agent $\mathcal{A}_{i}$; while $W_0(\cdot)$ is a Brownian motion representing a common noise shared by all agents (similar setup can be found in \cite{C14, HW14}). For $i=0,1,\cdots, N$, let $\mathbb{F}^{i}=\{{\mathcal F}_t^{i}\}_{0\les t\les T}$ be the natural filtration of $W_i(\cdot)$ augmented by all the ${\mathbb{P}}$-null sets. Then, $\mathbb{F}=\{{\mathcal F}_t\}_{0\les t\les T}=\{\sigma(\bigcup_{i=0}^N {\mathcal F}_t^{i})\}_{0 \les t \les T}$ is called the \emph{centralized information}.
$\sigma_0$ is unknown volatility but note that it might not be only $\mathbb{F}^{0}=\{\mathcal{F}^{0}_{t}\}_{0 \les t\les T}$-adapted with $\mathcal{F}^{0}$ the information generated by common noise $W_0(\cdot).$

\begin{remark}Individual diffusion part of \eqref{eq1} driven by $W_{i}$ does not include term like $Cx_{i}$ as in standard LQ control literature, mainly due to two concerns. First, introduction of $Cx_{i}$ will bring considerable technical difference in relevant analysis and we plan to address it in future work; second, the current setup is still rather general, especially including risky investments as its special case (i.e., $\sigma=0$). For simplicity, we assume that all the agents
have the same initial state. It is not hard to extend our results to the case that initial states of agents are i.i.d. random variables.
\end{remark}

When $D, D_{0} \neq 0,$ control process enters diffusion terms (driven by $W_i(\cdot), W_0(\cdot)$) of (\ref{eq1}), and in this case (\ref{eq1}) is said to be \emph{diffusion-controlled}. The study of diffusion-controlled systems has attracted extensive attentions, mainly because of their modeling power and application potentials in operational research and mathematical finance, etc. The readers may refer to \cite{YZ, BFY13, SY14} for relevant study of LQ diffusion-controlled systems, and related applications in mean-variance and portfolio selection problems. By comparison, the \emph{drift-controlled} (i.e., $D=D_{0}=0, B \neq 0)$ system is more classical in LQ literature and has been broadly-adopted in most mean-field game or team studies (e.g. \cite{HCM07, HCM12, WZ13}). Besides modeling, the diffusion-controlled system also differs from drift-controlled one in relevant analysis, for example, on the study of related Riccati equations and Hamiltonian systems.

Given state dynamics (\ref{eq1}), the cost functional of $\mathcal{A}_{i}$ is given by 
\begin{eqnarray}\label{eq2}
J_i(u)=&\displaystyle
\frac{1}{2}\mathbb{E}\int_0^T\Big\{\big|x_i(t)
-\Gamma x^{(N)}(t)-\eta(t)\big|^2_{Q}
+|u_i(t)|^2_{R}\Big\}dt+\frac{1}{2}\mathbb{E}|x_i(T)-\Gamma_0x^{(N)}(T)-\eta_0|_G^2,
\end{eqnarray}
where 
$x^{(N)}=\frac{1}{N}\sum_{j=1}^Nx_j$ is weakly-coupled state-average, and $u=\{u_1,
\ldots, u_N\} \in \mathbb{R}^{r \times N}$ the team strategy.
The admissible strategy set of  $\mathcal{A}_i$ is in the \emph{distributed} sense: $$
\begin{aligned}
{\cal U}_{i}^r =&\Big\{u_i(\cdot)\in\  L^{2}_{\mathbb{H}^{i}}(0, T; \mathbb{R}^{r}): \mathbb{H}^{i}=\{\mathcal{H}_{t}^{i}\}_{0 \leq t \leq T},  \mathcal{H}^{i}_{t}\triangleq\sigma\big\{{\mathcal F}_t^{0}\cup {\mathcal F}_t^{i}\cup\sigma(x_i(s),0 \leq s\leq t)\big\}\Big\}.
\end{aligned}
$$Here, $\{\mathcal{H}_{t}^{i}\}$ denotes the decentralized (or, distributed) information for the individual agent $\mathcal{A}_{i}$. Note that $x_{i}$ is not $\{\mathcal{F}^{i}_{t}\}$-adapted because of the state-average coupling $x^{(N)};$ thus, the inclusions of $\sigma(x_i(s))$ and ${\mathcal F}_t^{i}$ are both necessary in above formulation.
For comparison, the \emph{centralized strategy} set is:$$\begin{aligned}
{\cal U}_{c}^{r} &=\Big\{u_i(\cdot)\in  L^{2}_{\mathbb F}(0, T; \mathbb{R}^{r})\Big\}.
\end{aligned}
$$
Denote the social cost under volatility with soft constraint by
$$J_{\rm soc}^{(N)}(u,\sigma_0)=\sum_{i=1}^N\left(J_i(u)-\frac{1}{2}\mathbb{E}\int_0^T|\sigma_0(t)|^2_{R_0}dt\right)$$
with $R_{0}$ being the attenuation parameter of soft constraint (see \cite{BB95}).
The main goal of the current paper is to seek a set of distributed strategies to minimize the social cost under soft constraint for system (\ref{eq1})--(\ref{eq2}), i.e., $$\hbox{\textbf{(P)} minimize}_{u_{i}\in {\cal U}_{i}^r}J_{\rm soc}^{\rm wo}(u) \quad \text{with} \quad J_{\rm soc}^{\rm wo}(u)\stackrel{\Delta}{=}\sup_{\sigma_0\in {\cal U}^{n}_{c}}J_{\rm soc}^{(N)}(u, \sigma_0)$$over $\{u=(u_1,\cdots u_i, \cdots, u_{N}), u_i\in {\cal U}_{i}^r, i=1,\cdots,N\},$ where
$J_{\rm soc}^{\rm wo}(u)$ is the social cost under the worst-case volatility.

\ms

To simplify the analysis, we introduce the following hypothesis.

\ms

\textbf{(H1)} The state and cost functional coefficients satisfy:
$$\left\{
\begin{aligned}
&A,  C_0,\Gamma,  \Gamma_0\in \mathbb{R}^{n\times n}, \quad \quad \quad B, D, D_0\in\mathbb{R}^{n\times r}, \\
&Q\ges0,R>0, R_0>0, G\ges0, \quad  \quad f,\sigma,\sigma_0, \eta, \eta_0\in
L_{\mathbb{F}}^2(0,T,\mathbb{R}^n).
\end{aligned}\right.$$
Under (H1), by \cite{YZ}, for any $x_0$ and $u_i \in {\cal U}_c^r$, (\ref{eq1}) admits a unique strong solution $$\textbf{x}^{T}(\cdot)=(x_1^T(\cdot), \cdots, x_i^T(\cdot), \cdots, x^T_{N}(\cdot)) \in L^{2}_{\mathbb{F}}(\Omega;C([0,T];\mathbb{R}^{nN}))$$ with the following estimates hold true: for some $c_{1}$ independent of $N$,
$$\mathbb{E}\sup_{0 \les t \les T}|\textbf{x}(t)|^{2}
\les c_1 \mathbb{E} \Big[N|x_0|^{2}+N\big(\int_0^{T}|f(s)|ds\big)^{2}+N\int_0^{T}\big(|\sigma_0(s)|^2
+|\sigma(s)|^{2}\big)ds+\sum_{i=1}^{N}\int_0^{T}|u_{i}(s)|^{2}ds\Big].$$

\section {The control problem with respect to volatility uncertainty}\label{sec3}
From now on, the time variable $t$ might be suppressed when no confusion occurs. Let $u_i\in {\cal U}_{c}^r,  i=1,\cdots, N$ be fixed.
The optimal control problem with volatility uncertainty can be formulated as
$$\hbox{\textbf{(P1)} maximize}_{\sigma_0\in {\cal U}_{c}^n}J_{\rm soc}^{(N)}(u,\sigma_0)$$which is equivalent to:

\textbf{(P1$^\prime$)} minimize $ \check{J}_{\rm soc}^{(N)}$ over $\sigma_0\in{\cal U}_c^n$, where
\begin{align*}
\check{J}_{\rm soc}^{(N)}(\sigma_0)=&\frac{1}{2}\sum_{i=1}^N
\mathbb{E}\int_0^T\Big\{-\big|x_i
-\Gamma x^{(N)}-\eta\big|^2_{Q}
+|\sigma_0|^2_{R_0}\Big\}dt-\frac{1}{2}\mathbb{E}|x_i(T)-\Gamma_0x^{(N)}(T)-\eta_0|_G^2.
\end{align*}Hereafter, the following notations will be used to enable more compact representation.
Let 
$\textbf{u}=(u_1^T,\cdots,u_N^T)^T,\quad \textbf{1}=(1,\cdots,1)^T, \quad  {\mathbf{\sigma}_i}=(0,\cdots,0,\sigma^T, 0,\cdots,0)^T, \quad
\textbf{A}=\diag(A, \cdots,A), \quad \textbf{B}=\diag(B, \cdots,B), \textbf{D}_i\!\!=\!\diag(0, \cdots,0,D,0,\cdots,0),\textbf{C}_0\!=\!\diag(C_0, \cdots,C_0)$, $\textbf{D}_0\!\!=\!\diag(D_0, \cdots,D_0)$ and $\textbf{x}_0=(x_0^T,\cdots,x_0^T)$.

\begin{remark}\label{rem1}
\! \emph{Hereafter, whenever necessary, we may exchange the usage of notation $u=(u_1,\! \cdots,$ $ u_{N}) \in \mathbb{R}^{r \times N}$ and $\textbf{u}=(u_1^T,\cdots,u_N^T)^{T} \in \mathbb{R}^{rN}$ by noting they both represent the team decision profile among all agents, but only differ in formations.}
\end{remark}
With above notations, we can rewrite dynamics of all agents in a more compact form:
\begin{align*}\label{eq3}
	d\textbf{x}(t) =&\textbf{Ax}(t)dt+\textbf{Bu}(t)dt+ \textbf{1}\otimes f(t)dt+\sum_{i=1}^N[\textbf{D}_i\textbf{u}(t)+{\bf{\sigma}}_i(t)] dW_i(t)
	\cr&+[\textbf{C}_0\textbf{x}(t)+\textbf{D}_0\textbf{u}(t)+\textbf{1}\otimes \sigma_0(t)]dW_0(t),\ \textbf{x}(0)=\textbf{x}_0.
\end{align*}Recall $\otimes$ denotes the Kronecker product. Also, we introduce the following notations:
\begin{equation}\label{eq2a}
\left\{
\begin{aligned}
&\Xi_1:=\Gamma^TQ+Q\Gamma-\Gamma^TQ\Gamma, \quad\quad \Xi_2:=Q\eta-\Gamma^T Q\eta, \\
&\Xi_1^G:=\Gamma_0^TG +G\Gamma_0-\Gamma_0^TG\Gamma_0,\quad  \Xi^{G}_2:=G\eta_0-\Gamma_0^TG\eta_0.
\end{aligned}
\right.
\end{equation}
By rearranging the integrand of $\check{J}_{\rm soc}^{(N)}$, we have
\begin{align}
\label{eq4}
\check{J}^{(N)}_{\rm soc}
=&\frac{1}{2}\mathbb{E}\int_0^{T}\Big(-|\textbf{x}|_{\hat{\textbf{Q}}}^2+2\hat{\bf{\eta}}^T\textbf{x}
+{N}|\sigma_0|_{R_0}^2\Big)dt-\frac{1}{2}\mathbb{E}\big(|\textbf{x}(T)|^2_{\hat{\textbf{G}}}-2\hat{\bf{\eta}}^T_0\textbf{x}(T)\big),
\end{align}
where $\hat{\bf{\eta}}=\textbf{1}\otimes \Xi_2$, $\hat{\bf{\eta}}_0=\textbf{1}\otimes \Xi_2^G$, and $\hat{\textbf{Q}}=(\hat{Q}_{ij})$,
$\hat{\textbf{G}}=(\hat{G}_{ij})$ are given respectively by
\begin{align}\label{eq3a}
\hat{Q}_{ii}=Q-\Xi_1/N,\  \hat{Q}_{ij}=-\Xi_1/N,\
\hat{G}_{ii}=G-\Xi_1^G/N,\  \hat{G}_{ij}=-\Xi_1^G/N,\ 1\les i\ne j\les N.
\end{align}
Denote $${\bf{\Gamma}}_i=\Big[-\frac{\Gamma}{N}, \cdots, -\frac{\Gamma}{N},I-\frac{\Gamma}{N},-\frac{\Gamma}{N}, \cdots, -\frac{\Gamma}{N} \Big],$$
where $I-\frac{\Gamma}{N}$ is the $i$th element.
Note $\hat{\textbf{Q}}=\sum_{i=1}^N{\bf{\Gamma}}_i^T{Q}{\bf{\Gamma}}_i$. Then
$$\lambda_{min}(Q)\sum_{i=1}^N{\bf{\Gamma}}_i^T{\bf{\Gamma}}_i\les \hat{\textbf{Q}}\les
\lambda_{max}(Q)\sum_{i=1}^N{\bf{\Gamma}}_i^T{\bf{\Gamma}}_i.$$
For further analysis, we assume

\ms

(\textbf{H2}) $\check{J}_{\rm soc}^{(N)}(\sigma_0)$ of (P1$^\prime$) is convex in $\sigma_0$;

\ms

(\textbf{H2}$^{\prime}$) $\check{J}_{\rm soc}^{(N)}(\sigma_0)$ of (P1$^\prime$) is uniformly convex in $\sigma_0$.

\ms

We have the following equivalent conditions that ensure (H2).
\begin{proposition}
	The following statements are equivalent:
	
	{\rm (i)} $\check{J}_{\rm soc}^{(N)}(\sigma_0)$ is convex in $\sigma_0$; 
	
	{\rm (ii)} For any $\sigma_0\in {\cal U}_c^n$,
	$$\mathbb{E}\int_0^{T}\Big(-\emph{\textbf{z}}^T\hat{\bf{Q}}\emph{\textbf{z}}
	+{N}\sigma_0^TR_0\sigma_0\Big)dt-\mathbb{E}|\emph{\textbf{z}}(T)|_{\hat{\emph{\textbf{{G}}}}}^2\ges 0,$$
	where $\emph{\textbf{z}}\in \mathbb{R}^{nN}$ satisfies
	$$\left\{\begin{aligned}
	d\emph{\textbf{z}} =& \emph{\textbf{Az}}dt+(\emph{\textbf{C}}_0\emph{\textbf{z}}+\emph{\textbf{1}}\otimes\sigma_0)dW_0,\cr
	\emph{\textbf{z}}(0)&= 0.
	\end{aligned}\right.
	$$
	
	{\rm (iii)}  $\bar{J}^{\prime}_i(\sigma_0)$ is convex in $\sigma_0$,
	where
	$$\begin{aligned}
		\bar{J}^{\prime}_i(\sigma_0)\stackrel{\Delta}{=}&\mathbb{E}\int_0^T\Big\{-\big|(I-\Gamma) z_i\big|^2_{Q}+|\sigma_0|^2_{R_0}\Big\}dt-\mathbb{E}|(I-\Gamma_0)z_i(T)|_G^2
	\end{aligned}$$
	subject to
\begin{equation}\label{eq4c}
  dz_i(t) =Az_i(t)dt+[C_0z_i(t)+\sigma_0(t)]dW_0(t),\ z_i(0)=0.
\end{equation}
	
\end{proposition}

\emph{Proof.} (i) $\Leftrightarrow$ (ii) is given in \cite{HH16}.
From (\ref{eq4c}), we have $z_1=z_2=\cdots=z_N=z^{(N)}$.
Thus,
\begin{align}\label{eq4b}
&\mathbb{E}\int_0^{T}\Big(-|\textbf{z}|^2_{\hat{\textbf{Q}}}
+{N}|\sigma_0|^2_{R_0}\Big)dt-\mathbb{E}|\textbf{z}(T)|^2_{\hat{\textbf{G}}}\cr
=&
\sum_{i=1}^N\mathbb{E}\int_0^T\Big(-\big| z_i-\Gamma z_i\big|^2_{Q}+|\sigma_0|^2_{R_0}\Big)dt-\sum_{i=1}^N\mathbb{E}|(I-\Gamma_0)z_i(T)|_G^2\cr
=&N\Big[\mathbb{E}\int_0^T\Big(-\big|(I-\Gamma)z_i\big|^2_{Q}+|\sigma_0|^2_{R_0}\Big)dt-\mathbb{E}|(I-\Gamma_0)z_i(T)|_G^2\Big],
\end{align}
which implies that (ii) is equivalent to (iii). $\hfill \Box$

Denote $\hat{\textbf{1}}=\textbf{1}\otimes I$. By \cite{sly}, if the following Riccati equation
\begin{equation}\label{eq5a}
\begin{split}
&\dot{\textbf{P}}+\textbf{A}^T\textbf{P}+\textbf{P}\textbf{A}+\textbf{C}^T_0\textbf{P}\textbf{C}_0-\hat{\textbf{Q}}-
\big(\hat{\textbf{1}}^T\textbf{P}\textbf{C}_0\big)^T\big[NR_0+\hat{\textbf{1}}^T\textbf{P}\hat{\textbf{1}}\big]^{-1}\hat{\textbf{1}}^T\textbf{P}\textbf{C}_0=0,\cr & \textbf{P}(T)=-\hat{\textbf{G}}
\end{split}
\end{equation}
admits a solution such that $NR_0+\hat{\textbf{1}}^T\textbf{P}\hat{\textbf{1}}>0$, then
$\check{J}_{\rm soc}^{(N)}(\sigma_0)$  is uniformly convex,
which further gives that (H2$^{\prime}$) holds.
The above condition (\ref{eq5a}) is of high-dimension $nN \times nN$ which is not feasible to verify. Alternatively, we give the following necessary and sufficient condition with low-dimensionality.
\begin{proposition}
	The following statements are equivalent:
	
	{\rm (i)} $\check{J}_{\rm soc}^{(N)}(\sigma_0)$  is uniformly convex in $\sigma_0$;
	
	{\rm (ii)} $\bar{J}^{\prime}_i(\sigma_0)$  is uniformly convex in $\sigma_0$;
	
	{\rm (iii)} The equation
	\begin{equation}\label{eq8e}
	\dot{K}+KA+A^TK+C_0^TKC_0-C_0^TK(K+R_0)^{-1}KC_0+\Xi_1-Q=0, \ K=\Xi_1^G-G.
	\end{equation}
	admits a solution such that
	$K+R_0>0$.
\end{proposition}

\emph{Proof.} (i) By (\ref{eq4b}) and \cite{LZ99}, we have (i)$\Leftrightarrow$(ii).
(ii)$\Leftrightarrow$(iii)
is implied from \cite{sly}. $\hfill \Box$

By examining the variation of $ \check{J}_{\rm soc}^{(N)}$, we obtain the following result.

\begin{theorem}\label{thm3a}
	Suppose that $R_0>0$, then for any fixed admissible strategy set $u\!=\!(u_1, \cdots, u_{N})$ $\in \prod_{i=1}^N {\cal U}_i^r$, Problem \emph{(P1$^\prime$)} has a minimizer $\sigma_0^*(u)$  if and only if
	\emph{(H2)} holds and the following forward-backward equation system admits a solution $(x_i,p_i, \{\beta_{i}^{j}\}_{j=0}^{N})$:
	\begin{equation}\label{eq4a}
	\left\{
	\begin{aligned}
	dx_i=&(Ax_i\!+\!Bu_i\!+\!f)dt\!+\!(Du_i\!+\!\sigma)dW_i\!+\!\Big(C_0x_i\!+\!D_0u_i\!-\!\frac{R_0^{-1}}{N}\sum_{j=1}^N\beta_j^0\Big)dW_0,\\
	dp_i=&-[A^Tp_i+C_0^T\beta_i^0-Qx_i+\Xi_1x^{(N)}+\Xi_2 ]dt+\beta_i^0dW_0+\sum_{j=1}^N\beta_i^jdW_j,\\
	x_i(0)&={x}_0,\quad  p_i(T)=(-G)x_i(T)+\Xi_1^Gx^{(N)}(T)+\Xi_2^G,\quad i=1,\cdots,N.
	\end{aligned}\right.
	\end{equation}In this case, the minimizer $\sigma_0^*(u)=-\frac{R_0^{-1}}{N}\sum_{j=1}^N\beta_j^0$.
\end{theorem}

{\it Proof.} The ``if" part follows directly by the standard completion of square technique for (P1$^\prime$) and stationary condition reasoning for quadratic functional.

For ``only if" part, suppose $\sigma_0^*$ is a minimizer to Problem (P1$^\prime$). ${x}_i^*$ is the optimal state of agent $i$ under the volatility $\sigma_0^*$.
$x^{(N)}_*=\frac{1}{N}\sum_{j=1}^N x_j^*$. For $i=1,2,\cdots,N,$ denote $\delta x_i=x_i-{x}_i^*$ the increment of $x_{i}$ along with the variation $\delta \sigma_0=\sigma_0-\sigma_0^*.$ Similarly, $\delta x^{(N)}=\frac{1}{N}\sum_{j=1}^N\delta x_j$ and $\delta \check{J}_{\rm soc}^{(N)}(\sigma_0^*, \delta \sigma_0)= \check{J}_{\rm soc}^{(N)}(\sigma_0)- \check{J}_{\rm soc}^{(N)}(\sigma_0^*)+o(||\delta \sigma_{0}||_{L^{2}}),$ the Fr\'{e}chet differential of $\check{J}_{\rm soc}^{(N)}$ on $\sigma_0^*$ along with direction $\delta \sigma_0$.
By (\ref{eq1}),
\begin{equation}\label{eq4d}
d(\delta x_i)=A(\delta x_i)dt+[C_0(\delta x_i)+\delta \sigma_0]dW_0,\quad \delta x_i(0)=0, \quad i=1,2,\cdots,N.
\end{equation}
By standard variational principle, we have the following stationary condition on Fr\'{e}chet differential:
\begin{equation}\label{eq5}
\begin{aligned}
0=&\delta \check{J}_{\rm soc}^{(N)}(\sigma_0^*, \delta \sigma_0)\cr=&
\sum_{i=1}^N\mathbb{E}\int_0^T\Big\{\langle-Q\big[{x}^*_i-(\Gamma x^{(N)}_*+\eta)\big],\delta x_i-\Gamma\delta{x}^{(N)}\rangle +
\langle R_0\sigma_0^*,\delta \sigma_0\rangle \Big\}dt\cr
&+\sum_{i=1}^N\mathbb{E}\big\{\langle-G\big[{x}^*_i(T)-(\Gamma_0 x^{(N)}_*(T)+\eta_0)\big],\delta x_i(T)-\Gamma\delta{x}^{(N)}(T)\rangle\big\}.
\end{aligned}
\end{equation}
Introduce the adjoint equation
\begin{equation}\label{eq6}
\begin{split}
dp_i=\alpha_idt+\beta_i^0dW_0+\beta_i^idW_i+\sum_{j\not =i}\beta_i^jdW_j, \quad p_i(T)=(-G)x_i^*(T)+\Xi_1^Gx_*^{(N)}(T)+\Xi_2^G,
\end{split}
\end{equation}
where the processes $\{\alpha_i\}_{i=1}^{N},\{\beta_i^0\}_{i=1}^{N}$ and $\{\beta_i^j\}_{i \neq j}$ are to be determined. Then by It\^{o}'s formula,
\begin{align}\label{eq7}
&\mathbb{E}[\langle(-G)x_i(T)+\Xi_1^Gx^{(N)}(T)+\Xi_2^G,x_i(T)\rangle]
\cr=&\mathbb{E}\int_0^T\big[\langle \alpha_i,\delta x_i\rangle+\langle p_i,A\delta x_i\rangle+\langle \beta_i^0,\delta \sigma_0\rangle\big] dt.
\end{align}
It follows by (\ref{eq5})-(\ref{eq7}) that
\begin{align*}
0=&\mathbb{E}\sum_{i=1}^N\int_0^T\Big[\big\langle-Q\big( x^*_i-(\Gamma  x^{(N)}_*+\eta)\big),\delta x_i-\Gamma\delta{x}^{(N)}\big\rangle +
\langle R_0\sigma_0^*,\delta \sigma_0\rangle \Big]dt\cr
&+\sum_{i=1}^N \mathbb{E}\int_0^T\big[\langle \alpha_i,\delta x_i\rangle+\langle p_i,A\delta x_i\rangle+\langle \beta_i^0,C_0\delta x_i+\delta \sigma_0\rangle+\langle \beta_i^i,0\rangle\big] dt\cr
=&\mathbb{E}\int_0^T\Big\langle NR_0\sigma_0^*+\sum_{i=1}^N\beta_i^0,\delta \sigma_0\Big\rangle dt+\mathbb{E}\sum_{i=1}^N\int_0^T\Big\langle-Q\big[ x^*_i-(\Gamma x^{(N)}_*+\eta)\big]\cr
&+ {\Gamma^TQ} \big[(I-\Gamma)  x^{(N)}_*-\eta\big]+\alpha_i+A^Tp_i+C_0^T\beta_i^0, \delta x_i\Big\rangle dt,
\end{align*}
which leads to
$$\left\{\begin{aligned}
\alpha_i=&-\big[A^Tp_i+C_0^T\beta_i^0+ \Gamma^TQ \big[(I-\Gamma)  x^{(N)}_*-\eta\big]-Q\big[ x^*_i-(\Gamma  x^{(N)}_*+\eta)\big],\cr
\sigma_0^*=&-\frac{R_0^{-1}}{N}\sum_{i=1}^N\beta_i^0.
\end{aligned}
\right.$$
Thus, the Hamiltonian system (\ref{eq4a}) admits a solution $( x^*_i,{p}_i, \{\beta_{i}^{j}\}_{j=0}^{N})$. 
\hfill{$\Box$}



Let 
${p}^{(N)}=\frac{1}{N}\sum_{i=1}^N{p}_i$ and ${\beta}^{(N)}_0=\frac{1}{N}\sum_{i=1}^N{\beta}_i^0$.
It follows from (\ref{eq4a}) that state-average limits $\hat{x}=\lim_{N\longrightarrow +\infty}{x}^{(N)}$,  $\hat{p}=\lim_{N\longrightarrow +\infty}{p}^{(N)}$, and $\hat{\beta}_0=\lim_{N\longrightarrow +\infty}{\beta}^{(N)}_0$ satisfy:
\begin{equation}\label{eq8b}
\left\{
\begin{aligned}
d\hat{x}=&(A\hat{x}+B\hat{u}+f)dt+(C_0\hat{x}+D_0\hat{u}-R_0^{-1}\hat{\beta}_0)dW_0,\quad \\
d\hat{p}=&-\big[A^T \hat{p}+C_0^T\hat{\beta}_0-(Q-\Xi_1)\hat{x}+\Xi_2\big]dt+\hat{\beta}_0dW_0,\cr
\hat{x}(0)&=x,\quad \hat{p}(T)=(\Xi_1^G-G)\hat{x}(T)+\Xi_2^G.
\end{aligned}\right.
\end{equation}

%

%

\section {The control problem of agent $i$: person-by-person optimality}\label{sec4}
%
%

\subsection {Some variational analysis}\label{sec4.1}
When the 
volatility $\sigma_0^*=-\frac{R_0^{-1}}{N}\sum_{j=1}^N{\beta}_{j}^0$ is applied, we turn to study the outer minimization problem for team agents.

\textbf{(P2)}: minimize ${J}_{\rm soc}^{\rm wo}(u)$ over $\{u=(u_1,\cdots,u_N)|u_i\in {\mathcal U}_{c}^r\}$, where
\begin{align}\label{eq8d}
{J}_{\rm soc}^{\rm wo}(u)\triangleq &J^{(N)}_{\rm soc}(u, \sigma^{*}_{0}(u))\cr=&\frac{1}{2}\sum_{i=1}^N
\mathbb{E}\int_0^T\Big\{\big|x_i
-\Gamma x^{(N)}-\eta\big|^2_{Q}+|u_i|^2_R-|\sigma_0^*(u)|^2_{R_0}\Big\}dt\cr
&+
\frac{1}{2}\mathbb{E}|x_i(T)-\Gamma_0x^{(N)}(T)-\eta_0|_G^2,
\end{align}subject to\begin{equation}\label{eq9}
\left\{
\begin{aligned}
dx_i\!=&(Ax_i\!+\!Bu_i\!+\!f)dt\!+\!(Du_i\!+\!\sigma)dW_i
\!+\!\Big(C_0x_i\!+\!D_0u_i\!-\!\frac{R_0^{-1}}{N}\sum_{k=1}^N\beta_k^0\Big)dW_0,\\
dp_i\!=\!&-(A ^Tp_i+C_0^T\beta_i^0-Qx_i+\Xi_1x^{(N)}+\Xi_2)dt+\beta_i^0dW_0+\sum_{k=1}^N\beta_i^kdW_k,\\
x_i(0)&=x_0, \quad
p_i(T)=(-G)x_i(T)+\Xi_1^Gx^{(N)}(T)+\Xi_2^G.
\end{aligned}
\right.
\end{equation}
For further analysis, we introduce the following assumption.

\ms

\textbf{(H3)} ${J}_{\rm soc}^{\rm wo}(u)$ of (P2) is convex in $u$.

\ms

Suppose $\bar{u}=(\bar{u}_1,\cdots,\bar{u}_{i},\cdots,\bar{u}_N)$ and $\bar{x}=(\bar{x}_1,\cdots,\bar{x}_{i},\cdots,\bar{x}_N)$ are respectively the centralized optimal control and states of \textbf{(P2)} and we make the following person-by-person optimality variation around its optimal point. We now  perturb the control of $\mathcal{A}_{i}$ to be $u_i$ and keep $(\bar{u}_1,\cdots,\bar{u}_{i-1},\bar{u}_{i+1},\cdots,\bar{u}_N)$, the strategies of all other agents fixed. 
Let $\delta u_i=u_i-\bar{u}_i,$ and $\delta u_i\in {\cal U}_c^{r}$.
Denote $\delta x_j=x_j-\bar{x}_j$, $\delta p_j=p_j-\bar{p}_j$, and $\delta \beta_j^k=\beta_j^k-\bar{\beta}_j^k, j,k=1,\cdots,N$ the corresponding (forward, adjoint) state variation.
%
By (\ref{eq4a}) and
(\ref{eq9}), we have
\begin{equation}\label{eq11}
\left\{
\begin{aligned}
d(\delta x_i)\!=&(A\delta x_i\!+\!B\delta u_i)dt\!+\!(D\delta u_i)dW_i\!\cr&+\!\Big(C_0\delta x_i\!+\!D_0 \delta u_i\!-\!\frac{R_0^{-1}}{N} \sum_{k=1}^{N}\delta\beta_k^0\Big)dW_0,\ \delta x_i(0)=0,\\
d(\delta p_i)=&-\big(A^T \delta p_i+C_0^T\delta\beta_i^0-Q\delta x_i+\Xi_1\delta x^{(N)}\big)dt+\delta\beta_i^0dW_0\cr&+\delta\beta_i^idW_i+\sum_{k\neq i}\delta\beta_i^k dW_k,\ \delta p_i(T)=(-G)\delta x_i(T)+\Xi_1^G\delta x^{(N)}(T),\\
\end{aligned}
\right.
\end{equation}
and for $j\not= i$,
\begin{equation}\label{eq12}
\left\{
\begin{aligned}
d(\delta x_j)=&A\delta x_jdt+\Big(C_0\delta x_j-\frac{R_0^{-1}}{N} \sum_{l=1}^{N}\delta\beta_l^0\Big)dW_0,\ \delta x_j(0)=0,\\
d(\delta p_j)=& -\big(A^T \delta p_j+C_0^T\delta\beta_j^0-Q\delta x_j+\Xi_1\delta x^{(N)}\big)dt+\delta \beta_j^0dW_0+\delta\beta_j^jdW_j\cr&+\sum_{l\neq j}\delta\beta_j^l dW_l,\ \delta p_j(T)=(-G)\delta x_j(T)+\Xi_1^G\delta x^{(N)}(T).\\
\end{aligned}  \right.
\end{equation}
This implies that for any $j,j'\neq i$, $   \delta x_j=\delta x_{j'}$, which further gives
\begin{equation}\label{eq13}
\delta p_j=\delta p_{j'},\qquad \delta \beta_j^0=\delta \beta_{j'}^0, \qquad \hbox{for}\ \ j,j'\neq i.
\end{equation}
Let $\mathbb{E}_{{\mathcal F}^0}[\cdot]\stackrel{\Delta}{=}\mathbb{E}[\cdot|{\mathcal F}_t^0]$ (suppressing $t$).
Note that $W_j$ is independent of $W_0$. 
It follows from (\ref{eq11}) that
\begin{equation}\label{eq11b}
\left\{
\begin{aligned}
d(\mathbb{E}_{{\mathcal F}^0}[\delta x_i])=&\big(A\mathbb{E}_{{\mathcal F}^0}[\delta x_i]
+B\mathbb{E}_{{\mathcal F}^0}[\delta u_i]\big)dt\cr&+\Big(C_0\mathbb{E}_{{\mathcal F}^0}[\delta x_i]+D_0 \mathbb{E}_{{\mathcal F}^0}[\delta u_i]-\frac{R_0^{-1}}{N} \sum_{k=1}^{N}\mathbb{E}_{{\mathcal F}^0}[\delta\beta_k^0]\Big)dW_0,\\
d(\mathbb{E}_{{\mathcal F}^0}[\delta p_i])=&\!-\!\big(A^T \mathbb{E}_{{\mathcal F}^0}[\delta p_i]\!+\!C_0^T\mathbb{E}_{{\mathcal F}^0}[\delta\beta_i^0]\!-\!Q\mathbb{E}_{{\mathcal F}^0}[\delta x_i]\!+\!\Xi_1\mathbb{E}_{{\mathcal F}^0}[\delta x^{(N)}]\big)dt\\
&+\mathbb{E}_{{\mathcal F}^0}[\delta \beta_i^0]dW_0,\\
\mathbb{E}_{{\mathcal F}_0^0}[\delta x_i(0)]=&0,\quad
\mathbb{E}_{{\mathcal F}^0_T}[\delta p_i(T)]=(-G)\mathbb{E}_{{\mathcal F}^0_T}(\delta x_i(T))+\Xi_1^G\mathbb{E}_{{\mathcal F}^0_T}(\delta x^{(N)}(T)).
\end{aligned}  \right.
\end{equation}
It follows from (\ref{eq12}) that for $j\!\neq\!i$
\begin{equation}\label{eq12b}
\left\{
\begin{aligned}
d(\mathbb{E}_{{\mathcal F}^0}[\delta x_j])=&A\mathbb{E}_{{\mathcal F}^0}[\delta x_j]dt+\Big(C_0\mathbb{E}_{{\mathcal F}^0}[\delta x_j]-\frac{R_0^{-1}}{N} \sum_{k=1}^{N}\mathbb{E}_{{\mathcal F}^0}[\delta\beta_k^0]\Big)dW_0,\\
d(\mathbb{E}_{{\mathcal F}^0}[\delta p_j])=& -\big(A^T \mathbb{E}_{{\mathcal F}^0}[\delta p_j]+C_0^T\mathbb{E}_{{\mathcal F}^0}[\delta\beta_j^0]-Q\mathbb{E}_{{\mathcal F}^0}[\delta x_j]\\
&+\Xi_1\mathbb{E}_{{\mathcal F}^0}[\delta x^{(N)}]\big)dt+\mathbb{E}_{{\mathcal F}^0}[\delta \beta_j^0]dW_0,\\
\mathbb{E}_{{\mathcal F}_0^0}[\delta x_j(0)]\!=&0,
\mathbb{E}_{{\mathcal F}_T^0}[\delta p_j(T)]\!=(-G)\mathbb{E}_{{\mathcal F}_T^0}(\delta x_j(T))\!+\!\Xi_1^G\mathbb{E}_{{\mathcal F}_T^0}(\delta x^{(N)}(T)).
\end{aligned}  \right.
\end{equation}
Denote $\delta J^{\rm wo}_{\rm soc}(\bar{u}, \delta u_{i})$ the Fr\'{e}chet differential of $J^{\rm wo}_{\rm soc}$ at 
$\bar{u}$ along with direction
$\delta u_{i}$:
\begin{equation}\label{eq225}
J^{\rm wo}_{\rm soc}(\bar{u}\!+\!\delta u_{i})\!-\!J^{\rm wo}_{\rm soc}(\bar{u})\!=\!\delta J^{\rm wo}_{\rm soc}(\bar{u},\delta u_{i})\!+\!o(\|\delta u_{i}\|_{L^2})\!=\!\langle\mathcal{D}_{u_{i}}J^{\rm wo}_{\rm soc}(\bar{u}), \delta u_{i}\rangle\!+\!o(\|\delta u_{i}\|_{L^2})
\end{equation}where $\mathcal{D}_{u_{i}}J^{\rm wo}_{\rm soc}(\bar{u})$ is the Fr\'{e}chet derivative of $J^{\rm wo}_{\rm soc}$ at $\bar{u}$ with componentwise variation  $(0,\! \cdots\! \delta u_{i}^{T} \!\cdots\!, 0)$.
Then, from (\ref{eq13}), we can obtain
that for $j\not =i$,
\begin{equation*}
\begin{aligned}
&\delta {J}_{\rm soc}^{\rm wo}(\bar{u}, \delta u_i)\cr
= &\mathbb{E}\int_{0}^{T}\Big[\big<Q(\bar{x}_i\!-\!\Gamma\bar{x}^{(N)}\!-\!\eta),\delta x_i\big>\!+\!\big<Q\big(I\!-\!\frac{N-1}{N}\Gamma\big)\bar{x}^{(N)}\!-\!\frac{\bar{x}_i}{N}\!-\!\frac{N-1}{N}\eta),N \delta x_j\big>\\
&-\big<\frac{1}{N}\sum_{j\neq i}\Gamma^T Q(\bar{x}_j\!-\!\Gamma\bar{x}^{(N)}\!-\!\eta),\delta x_i\big>\!-\!\big<R_0^{-1}\bar{\beta}_0^{(N)},\delta \beta_i^0\big>\!-\!\big<R_0^{-1}\bar{\beta}_0^{(N)},(N-1)\delta \beta_j^0 \big>\\
&-\big<\Gamma^T Q\big((\bar{x}^{(N)}-\frac{\bar{x}_i}{N}-\frac{N-1}{N}\Gamma\bar{x}^{(N)}-\frac{N-1}{N}\eta\big),(N-1) \delta x_j\big>+\big<R\bar{u}_i,\delta u_i\big>\Big]dt\\
&+\mathbb{E}\big[\big<G(\bar{x}_i(T)-\Gamma_0\bar{x}^{(N)}(T)-\eta_0),\delta x_i(T)\big>-\big<\Gamma_0^TG(\bar{x}_i(T)-\Gamma_0\bar{x}^{(N)}(T)-\eta_0),\cr&\ \delta x^{(N)}(T)\big>+\sum_{j\neq i}\big<G(\bar{x}_j(T)-\Gamma_0\bar{x}^{(N)}(T)-\eta_0),\delta x_j(T)\big>\cr&-\sum_{j\neq i}\big<\Gamma_0^TG(\bar{x}_j(T)-\Gamma_0\bar{x}^{(N)}(T)-\eta_0),\delta x^{(N)}(T)\big>\big].
\end{aligned}
\end{equation*}When $N \to +\infty$, from (\ref{eq8b}), we further have
\begin{equation}\label{eq14b}
\begin{aligned}
&\lim_{N \to +\infty}
\delta {J}_{\rm soc}^{\rm wo}(\bar{u}, \delta u_i):=\delta \hat{J}_{i}(\bar{u}, \delta u_{i})=\langle \mathcal{D}_{u_{i}}\hat{J}_{i}(\bar{u}), \delta u_{i}\rangle
\cr=  &\mathbb{E}\int_{0}^{T}\Big[\big<Q\bar{x}_i,\delta x_i\big>-\big<Q(\Gamma\hat{x}+\eta)+\Gamma^T Q((I-\Gamma)\hat{x}-\eta),\delta x_i\big>+\big<R\bar{u}_i,\delta u_i\big>\\
&-\big<R_0^{-1}\hat{\beta}_0,\delta \beta_i^0\big>-\big<R_0^{-1}\hat{\beta}_0,\delta \beta ^{*}\big>\\
&+\big<Q((I-\Gamma)\hat{x}-\eta)-\Gamma^T Q((I-\Gamma)\hat{x}-\eta),\delta x^{*}\big>\Big]dt
\\
&+\mathbb{E}\big[\big<G\bar{x}_i(T),\delta x_i(T)\big>-\big<G(\Gamma_0\hat{x}(T)+\eta_0),\delta x_i(T)\big>\cr&+\big<G((I-\Gamma_0)\hat{x}(T)-\eta_0),\delta x^{*}(T)\big>
\!-\!\big<\Gamma_0^TG((I\!-\!\Gamma_0)\hat{x}(T)\!-\!\eta_0),\delta x_i(T)\big>\!\cr&-\!\big<\Gamma_0^TG((I\!-\!\Gamma_0)\hat{x}(T)\!-\!\eta_0),\delta x^{*}(T)\big>\big]
\end{aligned}
\end{equation}where $\delta \hat{J}_{i}(\bar{u}, \delta u_{i})$ is the Fr\'{e}chet differential of some auxiliary cost functional $\hat{J}_{i}$, to be constructed later (see (\textbf{P3}) in Section 5), $\mathcal{D}_{u_{i}}\hat{J}_{i}(\bar{u})$ the related Fr\'{e}chet derivative, and state average limits $(\hat{x}, \hat{\beta}_0)$ is to be determined by consistency condition in Section 5; moreover, for $j\not =i$, the quantities
\begin{equation*}\begin{aligned}
\delta x ^{*}:=N\mathbb{E}_{{\mathcal F}^0}[\delta x_j],\qquad \delta p^{*} :=N\mathbb{E}_{{\mathcal F}^0}[\delta p_j], \qquad  \delta \beta ^{*}:=N\mathbb{E}_{{\mathcal F}^0}[\delta \beta_j^0], 
\end{aligned}
\end{equation*}do not depend on $N$, and satisfy the following equations:
\begin{equation*}
\left\{
\begin{aligned}
d(\delta x^*)=&A(\delta x^*)dt+[C_0(\delta x^*)-R_0^{-1}(\delta \beta^*+\mathbb{E}_{{\mathcal F}^0}[\delta \beta_i^0])]dW_0,\ \delta x^*(0)=0,\\
d(\delta p^*)=&-\!\big[A^T(\delta p^*)\!+\!C_0^T(\delta \beta^*)\!-\!Q(\delta x^*)\!+\!\Xi_1(\mathbb{E}_{{\mathcal F}^0}[\delta x_i]\!+\!\delta x^*)\big]dt\cr&\!+\!(\delta \beta^*)dW_0,\ \delta p^*(T)=\Xi_1^G\mathbb{E}_{{\mathcal F}_T^0}(\delta x_i(T))-(G-\Xi_1^G)\delta x^*(T).\\
\end{aligned}
\right.
\end{equation*}
\begin{remark}
When studying the asymptotic behavior of \eqref{eq14b} with $N \longrightarrow +\infty,$ the following remainder term needs to be considered
$$\begin{aligned}
\epsilon_1^{(N)}:=&\mathbb{E}\int_{0}^{T}\Big[-\langle\Xi_1(\bar{x}^{(N)}-\hat{x}),\delta x_i\rangle-\langle R_0^{-1}(\bar{\beta}_0^{(N)}-\hat{\beta}_0),\delta \beta_i^0+(N-1)\delta \beta_j^0 \rangle\cr
& +\langle (Q-\Xi_1)(\bar{x}^{(N)}-\hat{x}),N\delta x_j\rangle \Big]dt -\mathbb{E}[\langle\Xi_1^G(\bar{x}^{(N)}(T)-\hat{x}(T)),\delta x_i(T)\rangle]\cr
&+\mathbb{E}[\langle (G-\Xi_1^G)(\bar{x}^{(N)}(T)-\hat{x}(T)),N\delta x_j(T)\rangle].
\end{aligned}$$
Because $\|\delta u_i\|_{L^{2}}<\infty$, $\epsilon_1^{(N)}$ should be an infinitesimal term with same order to $\|\bar{x}^{(N)}-\hat{x}\|_{\max}+\|\bar{\beta}^{(N)}_0-\hat{\beta}_0\|_{L_2}$ ($N\to \infty$). Actually,  from (\ref{eq4a}) and (\ref{eq8b}) we may obtain $\|\bar{x}^{(N)}-\hat{x}\|_{\max}^2+\|\bar{\beta}^{(N)}_0-\hat{\beta}_0\|_{L_2}^2=O(\frac{1}{{N}})$ (The rigorous proof will be given in Section \ref{sec8}). Thus, $\epsilon_1^{(N)}=O(\frac{1}{\sqrt{N}})\|\delta u_i\|_{L^{2}}\|\bar{u}\|_{L^{2}}.$ 
	
\end{remark}

\subsection{Duality derivation}A key point in analyzing the social optimization problem is to formulate some auxiliary control problem for a given agent, based on
$\delta \hat{J}_{i}=\lim_{N \to +\infty}
\delta {J}_{\rm soc}^{\rm wo}$ of \eqref{eq14b}, thus the decentralized strategy can be derived via some mean-field game procedure. Such auxiliary problem can be derived via some variational analysis (see \cite{WH17} for related variational analysis but with only drift-controlled term). Due to volatility uncertainty, all states of agents are coupled via some high-dimensional FBSDE system. Therefore, related variational analysis becomes fairly different to that of \cite{WH17}, and depends on a \emph{two-step duality} procedure, as discussed below.\\

\emph{Step 1.} (Duality independent of $(\delta x^*, \delta p^*)$). \\ The first step removes the dependence of $\delta \hat{J}_{i}(\bar{u}, \delta u_{i})$ on $(\delta x^*, \delta p^*),$ the variational process common to all agents. To this end, introduce the adjunct FBSDE: 
\begin{equation}\label{d1}
\left\{\begin{aligned}
dy & = f_{0} dt +zdW_0(t), \quad y(T)=\left(G-\Xi_1^G\right)\hat{x}(T)-\Xi_2^G,\\
dh & = f_1 dt + f_2 dW_0(t), \quad h(0)=0
\end{aligned}
\right.
\end{equation}where the drivers $(f_0; f_1, f_2)$ are to be determined.
Note $h(0)=0$, and
$$\delta p^*(T)-\Xi_1^G \mathbb{E}_{\mathcal{F}_T^0}(\delta x_i) (T) -(G-\Xi_1^G)\delta x^*(T)= 0.$$
By It\^{o}'s formula,
\begin{equation}\label{d2}
\begin{aligned}
0= & \mathbb{E}\int_0^T\Big\{\left\langle h, -\left(A^T \delta p^*+C_0^T(\delta\beta^*)+\Xi_1 \mathbb{E}_{{\mathcal F}^0}(\delta x_i)-(Q-\Xi_1)(\delta x^*)\right) \right. \\
&- \Xi_1^G \big(A \mathbb{E}_{{\mathcal F}^0}(\delta x_i)+B\mathbb{E}_{{\mathcal F}^0}(\delta u_i)-(G-\Xi_1^G)A(\delta x^*) \big) \big\rangle \cr& +\left\langle \delta p^*-\Xi_1^G \mathbb{E}_{{\mathcal F}^0}(\delta x_i)-(G-\Xi_1^G)\delta x^*(t) ,f_1 \right\rangle \\
&+\left\langle f_2, \delta \beta^*-\Xi_1^G \left(C_0 \mathbb{ E}_{{\mathcal F}^0}(\delta x_i) +D_0\mathbb{E}_{{\mathcal F}^0}(\delta u_i)\right)\right.\cr
&\left.-(G-\Xi_1^G)\left[C_0(\delta x^*)-R_0^{-1}\left(\delta \beta^*+ \mathbb{E}_{{\mathcal F}^0}(\delta \beta_i^0)\right)\right] \right\rangle \Big\}dt.\\
\end{aligned}
\end{equation}
Using It\^{o} formula to $\langle \delta x^*, y \rangle$, we have
\begin{equation*}
\begin{aligned}
&\mathbb{E} \left\langle (G-\Xi_1^G)\hat{x}(T)-\Xi_2^G,\delta x^*(T)\right\rangle\\
=\ &\mathbb{E} \big[\int_{0}^{T} \langle \delta x^*, f_0 \rangle +\langle A^T y, \delta x^* \rangle +\langle z, C_0 \delta x^* -R_0^{-1}(\delta \beta^* + \mathbb{E}_{{\mathcal F}^0}(\delta \beta_i^0)) \rangle\big] dt.\\
\end{aligned}
\end{equation*}
It follows from (\ref{d1}) and (\ref{d2}) that
\begin{equation}\label{d3}
\begin{aligned}
&\mathbb{E} \left\langle(G-\Xi_1^G)\hat{x}(T)-\Xi_2^G ,\delta x^*(T)\right\rangle\\
=\ &\mathbb{E} \int_{0}^{T} \Big[\big\langle \delta x^*, f_0+A^T y +C_0^T z+(Q-\Xi_1)^Th -A^T(G-\Xi_1^G)^Th\cr& -(G-\Xi_1^G)^T f_1 -C_0^T(G-\Xi_1^G) f_2 \big\rangle+\langle \delta p^*, -A h +f_1 \rangle\\
&+\langle \delta \beta^*, -R_0^{-1} z -C_0h +f_2 +R_0^{-1}(G-\Xi_1^G)f_2 \rangle \\
&+\langle \mathbb{E}_{{\mathcal F}^0}(\delta u_i),\!-B^T \Xi_1^G h\!-\!D_0^T\Xi_1^Gf_2 \rangle\!+\!\langle \mathbb{E}_{{\mathcal F}^0}(\delta \beta_i^0),-R_0^{-1}z\!+\!R_0^{-1}(G\!-\!\Xi_1^G)f_2 \rangle\\
&+ \langle \mathbb{E}_{{\mathcal F}^0}(\delta x_i),-\Xi_1^T h - A^T (\Xi_1^G)^T h -(\Xi_1^G)^T f_1 -C_0^T \Xi_1^G f_2 \rangle\Big]dt.
\end{aligned}
\end{equation}Set \begin{equation}\label{eq4444}\mathbb{I}^G :=\left[I+R_0^{-1}(G-\Xi_1^G)\right]^{-1}\end{equation}(note that $\mathbb{I}^G=I$ if $G=0$). Comparing the coefficients, we obtain
\begin{equation}\label{eq26}
\left\{
\begin{aligned}
f_1=&Ah, \quad f_2=\mathbb{I}^G(R_0^{-1}z +C_0 h +R_0^{-1}\hat{\beta}_0)\\
f_0=&-(A^T y +C_0^T z + (Q-\Xi_1)h - A^T(G-\Xi_1^G)h -(G-\Xi_1^G)f_1\cr
& -C_0^T (G-\Xi_1^G)f_2  +(Q-\Xi_1) \hat{x} -\Xi_2).\\
\end{aligned}\right.
\end{equation}
Then, we have
\begin{equation}\label{eq26aa}
\left\{
\begin{aligned}
dy= &-\Big[A^T y +C_0^T z + (Q-\Xi_1)h - A^T(G-\Xi_1^G)h -(G-\Xi_1^G)Ah \\
&-C_0^T(G\!-\!\Xi_1^G) \mathbb{I}^G (R_0^{-1}z\!+\!C_0 h\!+\!R_0^{-1} \hat{\beta}_0)\!+\!(Q\!-\!\Xi_1)\hat{x}\!-\!\Xi_2\Big]dt\!+\!z dW_0(t)\\
dh= &Ah dt + \mathbb{I}^G (R_0^{-1}z+C_0 h + R_0^{-1} \hat{\beta}_0)dW_0(t)\\
y(T)&= (G-\Xi_1^G)\hat{x}(T)-\Xi_2^G,\quad h(0)=0.
\end{aligned}\right.
\end{equation}
Let $\xi_1=(Q-\Xi_1)\hat{x} -\Xi_2$, and\ $\xi_2 =-R_0^{-1}\hat{\beta}_0$.
From (\ref{d3}), we obtain
\begin{equation*}
\begin{aligned}
&\mathbb{E}\left\langle (G-\Xi_1^G)\hat{x}(T)-\Xi_2^G, \delta x^*(T) \right\rangle
+\mathbb{E}\int_{0}^{T} [\langle\delta x^*, \xi_1 \rangle +\langle \delta \beta^*, \xi_2 \rangle] dt \\
=\ & \mathbb{E} \int_{0}^{T}  \Big[\left\langle \mathbb{E}_{{\mathcal F}^0}(\delta \beta_i^0), -R_0^{-1}z +R_0^{-1}(G-\Xi_1^G)f_2 \right\rangle+\left\langle  \mathbb{E}_{{\mathcal F}^0}(\delta u_i),-B^T \Xi_1^G h -D_0^T \Xi_1^G f_2 \right\rangle\\
& +\left\langle \mathbb{E}_{{\mathcal F}^0}(\delta x_i), -\Xi_1 h-A^T \Xi_1^G h -\Xi_1 f_1 -C_0^T \Xi_1^G f_2 \right\rangle\Big]dt.
\end{aligned}
\end{equation*}
 Then a direct computation from (\ref{eq14b}) shows that\begin{equation}\label{eq30}
\begin{aligned}
&\delta \hat{J}_{i}(\bar{u}, \delta u_{i}))=\lim_{N \rightarrow+\infty} \delta J_{\rm soc}^{\rm wo}(\bar{u}, \delta u_{i})\cr=\ &\mathbb{E} \int_{0}^{T} \big[\langle Q\bar{x}_i,\delta x_i \rangle\!+\!\langle R\bar{u_i},\delta u_i \rangle\!-\!\langle \Xi_1 \hat{x}\!+\!\Xi_2, \delta x_i \rangle \!-\!\langle R_0^{-1} \hat{\beta}_0, \delta \beta_i^0 \rangle \big]dt \\
&+\mathbb{E }[\langle G\hat{x}_i(T), \delta x_i(T) \rangle - \langle \Xi_1^G \hat{x}(T) +\Xi_2^G, \delta x_i(T) \rangle] \\
&+  \mathbb{E} \int_{0}^{T} \Big[\langle -R_0^{-1} z +R_0^{-1} (G-\Xi_1^G) \mathbb{I}^G (R_0^{-1} z +C_0 h +R_0^{-1}\hat{\beta}_0),\delta \beta_i^0  \rangle\\
&+ \langle \delta x_i, -\Xi_1 h\!-\!A^T \Xi_1^G h\!-\!\Xi_1^G A h \!-\! C_0^T \Xi_1^G f_2 \rangle\! -\!\langle \delta u_i, B^T \Xi_1^G h
+D_0^T \Xi_1^G f_2 \rangle\Big]dt.
\end{aligned}
\end{equation}
Let
$\xi_3=-R_0^{-1}(z+\hat{\beta}_0)+R_0^{-1}(G-\Xi_1^G)f_2$. Then we will consider the following term in step 2,
\begin{equation*}
\begin{aligned}
&\mathbb{E}\int_{0}^{T} \big[-\langle R_0^{-1} \hat{\beta}_0, \delta \beta_i^0 \rangle + \langle -R_0^{-1} z +R_0^{-1}(G-\Xi_1^G)f_2, \delta \beta_i^0  \rangle \big]dt =\mathbb{E} \int_{0}^{T} \langle \xi_3, \delta \beta_i^0 \rangle dt.
\end{aligned}
\end{equation*}

\emph{Step 2.} (Duality independent of $(\delta \beta_i^{0}$)). \\ The second step removes the dependence of $\delta \hat{J}_{i}(\bar{u}, \delta u_{i})$ on \emph{backward} variational process $\delta \beta_i^{0}$. Thus, the derived auxiliary problem will end up with a \emph{forward} LQ control on $(\delta u_i, \delta x_i$) only.
To this end, first introduce the adjoint process
\begin{equation*}
\begin{aligned}
&d \Phi =g_1 dt + g_2 dW_0(t),\ \Phi(0)=0,
\end{aligned}
\end{equation*} with $g_1, g_2$ to be determined.
Note\begin{equation*}\left\{
\begin{aligned}
&d \delta p_i =-\left(A^T \delta p_i +C_0^T \delta \beta_i^0 - Q \delta x_i + \Xi_1 \delta x^{(N)}\right)dt + \delta \beta_i^0 dW_0 + \sum_{k=1}^{N} \delta \beta_i^k dW_k(t)\\
& \delta p_i(T)=(-G)\delta x_i(T) +\Xi_{1}^{G} \delta x^{(N)}(T).
\end{aligned}\right.
\end{equation*}
Then, by It\^{o}'s formula, we obtain\begin{equation}\label{d5}
\begin{aligned}
0= \mathbb{ E} \int_{0}^{T} &\big[\langle \delta p_i, -A \Phi +g_1 \rangle +\langle \delta \beta _i^0, -C_0 \Phi +g_2 \rangle \\
&\!+\!\langle  \delta x_i, Q^T \Phi\! +\!A^T G \Phi\!+\!G g_1 \!+\!C_0^T G g_2 \rangle-\langle \delta u_i, B^T G \Phi + D_0^T G g_2 \rangle\big] dt\end{aligned}\end{equation}which implies
$g_1=A \Phi$, and $g_2=C_0 \Phi -\xi_2=C_0 \Phi+R_0^{-1}(z+\hat{\beta
})-R_0^{-1}(G-\Xi_1^G) f_2$.
Then we have\begin{equation*}
\begin{aligned}
&d \Phi = A \Phi dt +\big[ C_0 \Phi +R_0^{-1}(z+ \hat{\beta}_0) - R_0^{-1}(G-\Xi_1^G)f_2 \big]dW_0(t), \Phi(0)=0,
\end{aligned}
\end{equation*}
which with (\ref{eq26}) gives $\Phi=h.$
From (\ref{d5}),
\begin{equation*}
\mathbb{E}\! \int_{0}^{T}\! \langle  \delta \beta_i^0, \xi_3 \rangle dt\! =\!\mathbb{ E} \int_{0}^{T}\! \big[\langle \delta x_i, Q^T \Phi +A^T G \Phi +G g_1\!+C_0^T G g_2 \rangle +\langle \delta u_i, B^T G \Phi \!+\!D_0^TG g_2 \rangle\big] dt.
\end{equation*}
From this with (\ref{eq30}), the variational functional 
becomes
\begin{equation}\label{eq20d}
\begin{split}
\delta \hat{J}_{i}(\bar{u}, \delta u_{i})=&\mathbb{ E} \int_{0}^{T}\Big[\langle Q\bar{x}_i, \delta x_i \rangle \!+\!\langle  R \bar{u}_i, \delta u_i \rangle \\
& \!-\! \langle \Xi_1 \hat{x} \!+\!\Xi_2, \delta x_i \rangle\!+\!\left\langle Q^T h \!+\! A^T G h \!+\! G A h \!+\! C_0^T G g_2, \delta x_i  \right \rangle
\\
&\!-\! \langle\Xi_1 h \!+\! A^T \Xi_1^G h\!+\!\Xi_1^G A h+C_0^T \Xi_1^G f_2,\delta x_i \rangle \cr&+ \langle B^T G h \!+\!D_0^T G g_2, \delta u_i \rangle\!+\! \langle\! -\!B^T \Xi_1^G h \!-\!D_0^T G g_2, \delta u_i \rangle\Big] dt\\
&+ \mathbb{E}[\langle G \bar{x}_i(T), \delta x_i (T) \rangle -\langle \Xi_i^G \hat{x}(T)+\Xi_2^G,\delta x_i(T) \rangle]
\end{split}
\end{equation}where $\begin{aligned}
g_2 = f_2 = \mathbb{I}^G \left( R_0^{-1}(z+\hat{\beta}_0)+C_0 h \right).
\end{aligned}$

\section{Decentralized robust team strategy design }\label{sec5}

By (\ref{eq8b}) and limiting social variational functional (\ref{eq20d}),
we construct the following auxiliary control problem, for a representative agent, still indexed by $\mathcal{A}_{i}$.\\

\textbf{(P3)}: minimize $\hat{J}_i(u_i)$ over $u_i\in \mathcal{U}_i$, with state dynamics and cost functional:
\begin{equation}\label{eq21}
\begin{aligned}
dx_i\!\!=&(Ax_i\!\!+\!\!Bu_i\!+\!f)dt\!+\!(Du_i\!+\!\sigma) dW_i\!+\!(C_0x_i\!+\!\!D_0u_i\!-\!R_0^{-1}\hat{\beta}_0)dW_0,\ x_i(0)=x_0,\cr
\hat{J}_i(u_i)\!\!=\!&\frac{1}{2} \mathbb{E}\Bigg\{\int_{0}^{T}| x_i|^2_Q+| u_i|^2_{R}-2 \langle \Xi_1\hat{x}+\Xi_2, x_i\rangle+2 \langle (Q-\Xi_1)h,x_i \rangle \\
& + 2\langle A^T Gh +G A h+C_0^T G g_2-A^T \Xi_1^G h-\Xi_1^G Ah -C_0^T \Xi_1^G g_2, x_i \rangle \\
&+ 2\langle  B^T (G-\Xi_1^G) h +D_0^T (G-\Xi_1^G g_2), u_i\rangle dt\\
&+ \langle Gx_i(T),x_i(T) \rangle -2 \langle \Xi_1^G \hat{x}(T) +\Xi_2^G,x_i(T) \rangle  \Bigg\}.
\end{aligned}
\end{equation}Here the triple $(\hat{x},\hat{\beta}_0, h)$ satisfies the following limiting (off-line) system parameterized by undetermined process $\hat{u}:$\begin{equation}\label{eq22}
\left\{
\begin{aligned}
d\hat{x}=&(A\hat{x}+B\hat{u}+f)dt+(C_0\hat{x}+D_0\hat{u}-R_0^{-1}\hat{\beta}_0)dW_0,\qquad \\
d\hat{p}=&-\left(A^T \hat{p}+C_0^T\hat{\beta}_0-Q\hat{x}+\Xi_1\hat{x} + \Xi_2\right)dt+\hat{\beta}_0dW_0,\ \quad \\
dy=&-\left(A^T y\!+\!C_0^Tz\!+\!(Q\!-\!\Xi_1)h\!+\!(Q\!-\!\Xi_1)\hat{x}\!-\!\Xi_2\right)dt\!+\!(A^T (G\!-\!\Xi_1^G) h \\
& \left.+(G-\Xi_1^G)A h  +C_0^T(G-\Xi_1^G) f_2\right)dt +zdW_0,\qquad   \\
dh=&Ahdt+\left( I+ R_0^{-1}(G-\Xi_1^G) \right)^{-1}\left(C_0h+R_0^{-1}(z+\hat{\beta}_0)\right)dW_0,
\cr \hat{x}(0)\!=&x_0,\ \hat{p}(T)\!=\! \left(\Xi_1^G\! -\!G \right) \hat{x}(T) \!+\!\Xi_2^G,\ y(T)\!=\!(G\!-\!\Xi_1^G)\hat{x}(T)\!-\!\Xi_2^G,\ h(0)\!=\!0.
\end{aligned}\right.
\end{equation}
\begin{remark}\label{rem5.1}
 FBSDE \eqref{eq22} can be decomposed into subsystems $(\hat{x}, \hat{p}, \hat{\beta}_0)$ and $(h, y, z)$ which are decoupled for each other. Thus, solvability of \eqref{eq22} reduces to that of $(h, y, z)$ and $(\hat{x}, \hat{p}, \hat{\beta}_0)$ separately. Section $6$ will discuss the global solvability of subsystem $(h, y, z)$, and similar analysis can be applied to $(\hat{x}, \hat{p}, \hat{\beta}_0)$ considering these two subsystems have similar coupling structures. Moreover, parameter process $\hat{u}$ will be further determined by some consistency condition system through mean-field game argument. \end{remark}
Let $\hat{u}(t)\in {\mathcal F}_t^{0}$ be fixed. We study the decentralized open-loop strategy and related consistency condition system. We have the following result by maximum principle.
\begin{theorem}\label{thm4}
	Suppose that 
	$Q\ges0, G\ges0$ and $R>0$.
	Then the following backward stochastic differential equation (BSDE) admits a (unique) solution
	
	\begin{equation}\label{eq23}
	\begin{aligned}
	dk_i\!=\!&-\big[A^T k_i\!+\!C_0^T\zeta_0\!+\!Qx_i\!-\!\Xi_1\hat{x}\!+\!(Q\!-\!\Xi_1)h\!-\!\Xi_2\!+\!\mathcal{K}(G,g_2)\!-\!\mathcal{K}(\Xi_1^G,g_2) \big]dt\\
	&+\zeta_0dW_0+\zeta_idW_i,\qquad k_i(T)=Gx_i(T)-\Xi_1^G \hat{x}(T)-\Xi_2^G,
	\end{aligned}
	\end{equation}where $\mathcal{K}(G,g_2)=A^T Gh+GAh+C_0^TGg_2$, $\mathcal{K}(\Xi_1^G,g_2)=A^T\Xi_1^G h+\Xi_1^G Ah +C_0^T \Xi_1^G g_2$
	and\begin{equation}\label{eq24}
	\check{u}_i\!=\!-R^{-1}(B^Tk_i\!+\!D_0^T\zeta_0\!+\!D^T\zeta_i\! +\! B^T (G\!-\!\Xi_1^G)h\!+\!D_0^T (G\!-\!\Xi_1^G)g_2), \quad i\!=\!1,\cdots,N.
	\end{equation}
\end{theorem}

\emph{Proof.}  Since $Q\ges0, G\ges0$ and $R>0$, (P3) is uniformly convex, which implies the unique solvability of (P3).
Assume that $\check{u}_i$ is the unique optimal control of Problem (P3) and $\check{x}_i$ is the state equation under $\check{u}_i$. Then
\begin{equation}\label{eq25}
\begin{aligned}
0=&\delta {\hat{J}}_i(\check{u}_i, \delta u_{i})\cr=&\mathbb{E}\int_{0}^{T}\big[\big<Q\check{x}_i,\delta x_i\big>+\big<R\check{u}_i,\delta u_i\big>-\big<\Xi_1\hat{x}+\Xi_2,\delta x_i\big>+\big<(Q-\Xi_1)h,\delta x_i\big>\big]\cr
& + \big<\mathcal{K}(G,g_2)- \mathcal{K}(\Xi_1^G,g_2),\delta x_i\big> +\big< B^T(G-\Xi_1^G)h+D_0^T(G-\Xi_1^G)g_2,\delta u_i \big>dt\\
&+\big< G \check{x}_i(T),\delta x_i(T)\big>-\big< \Xi_1^G\hat{x}(T)+\Xi_2^G,\delta x_i(T)\big>.
\end{aligned}
\end{equation}Given $\hat{x}$ and $h$, (\ref{eq23}) is a standard linear BSDE and thus has a unique solution $(k_i,\zeta_0,\zeta_i)$. Then
\begin{equation}\label{eq37}
\begin{aligned}
&\langle G \check{x}_i(T)-\Xi_1^G \hat{x}(T)-\Xi_2^G,\delta x_i(T) \rangle\cr
=&\mathbb{E}\int_{0}^{T}\Big\{\langle -(Qx_i-\Xi_1\hat{x}-(Q-\Xi_1)h-\Xi_2-\mathcal{K}(G,g_2)+\mathcal{K}(\Xi_1^G,g_2)),\delta x_i\rangle\\
&+\langle k_i,B\delta u_i\rangle+\langle \zeta_0,D_0\delta u_i\rangle
+\langle \zeta_i,D\delta u_i\rangle\Big\}dt.
\end{aligned}
\end{equation}
From this and (\ref{eq25}), we have
\begin{equation}\label{eq25b}
0=\mathbb{E}\int_{0}^{T}\langle R\check{u}_i+ B^T(G-\Xi_1^G)h+D_0^T(G-\Xi_1^G)g_2+B^T k_i+ D_0^T\zeta_0+ D^T\zeta_i,\delta u_i\rangle dt,
\end{equation}
which implies the open-loop optimal strategy:
$$\begin{aligned}
\check{u}_i=&-R^{-1}\left(B^Tk_i+D_0^T\zeta_0+D^T\zeta_i+B^T (G-\Xi_1^G)h+D_0^T (G-\Xi_1^G)g_2\right)\cr
=&-R^{-1}(v_{i}+B^T (G-\Xi_1^G)h+D_0^T (G-\Xi_1^G)g_2).
\end{aligned}$$Note that here,\begin{equation}\label{DD}v_i:=B^T k_i+D^T \zeta_i+D_0 \zeta_0.\end{equation} \hfill $\Box$

After the strategy  (\ref{eq24}) is applied, we obtain the following state equation:
\begin{equation*}
\begin{aligned}
dx_i= &\left[ Ax_i -BR^{-1} \left( v_i +B^T(G-\Xi_1^G)h+D_0^T(G-\Xi_1^G)g_2\right)+ f\right]dt\\
&+\left[- DR^{-1} \left( v_i +B^T(G-\Xi_1^G)h+D_0^T(G-\Xi_1^G)g_2\right)+ \sigma\right]dW_i\\
&+\left[C_0x_i-DR^{-1} \left( v_i +B^T(G-\Xi_1^G)h+D_0^T(G-\Xi_1^G)g_2\right)-R_0^{-1}\hat{\beta}_0\right]dW_0.
\end{aligned}
\end{equation*}Consequently, consistency argument implies the following consistency condition (CC) system to $(\hat{x},\hat{\beta}_0, h)$:
\begin{equation}\label{eq43}
\left\{ \begin{aligned}
dx_i\!=\!&\left[Ax_i-BR^{-1}(v_i+B^T(G-\Xi_1^G)h+D_0^T(G-\Xi_1^G)g_2)+f\right]dt\\
&\!+\!\left[ C_0x_i-D_0R^{-1}(v_i+B^T(G-\Xi_1^G)h+D_0^T(G-\Xi_1^G)g_2)-R_0^{-1}\hat{\beta}_0 \right]dW_0\\
&\!-\!\left[ D R^{-1}(v_i+B^T(G-\Xi_1^G)h+D_0^T(G-\Xi_1^G)g_2)-\sigma \right]dW_i,\quad \\
dk_i\!=\!&-\!\big[A^T k_i\!+\!C^T_0\zeta_0\!+\!Qx_i\!-\!\Xi_1 \mathbb{E}_{{\mathcal F}^0}[x_i] \!-\!\Xi_2\!+\!(Q\!-\!\Xi_1)h\!\\
&+\!\mathcal{K}(G,g_2)\!-\!\mathcal{K}(\Xi_1^G,g_2)\big]dt\!+\!\zeta_0dW_0+\zeta_idW_i, \quad\\
d\hat{p}\!=\!&-\big[A^T \hat{p}+C_0^T\hat{\beta}_0-(Q-\Xi_1)\mathbb{E}_{{\mathcal F}^0}[x_i]+\Xi_2\big]dt+\hat{\beta}_0dW_0, \ \\
dy\!=\!&-\big[A^T y+C_0^Tz+(Q-\Xi_1)h+(Q-\Xi_1)\mathbb{E}_{\mathcal{F}_T^0}[x_i]-\Xi_2\big]dt\\
&+[A^T(G-\Xi_1^G)h\!+\!(G-\Xi_1^G)Ah+C_0^T(G-\Xi_1^G)g_2]dt+zdW_0,\ \\
dh\!=\! &Ahdt+[I+R_0^{-1}(G-\Xi_1^G)]^{-1}[C_0h+R_0^{-1}(z+\hat{\beta}_0)]dW_0,\\
x_i(0)\!=&x_0, k_i(T)\!=\!Gx_i(T)\!-\!\Xi_1^G \mathbb{E}_{\mathcal{F}_T^0}[x_i(T)]\!-\!\Xi_2^G,\hat{p}(T)\!=\!(\Xi_1^G\!-\!G)\mathbb{E}_{\mathcal{F}_T^0}[x_i(T)]\!+\!\Xi_2^G,\cr y(T)=&(G-\Xi_1^G)\mathbb{E}_{\mathcal{F}_T^0}[x_i(T)]-\Xi_2^G,\ h(0)=0.
\end{aligned} \right.
\end{equation}\begin{remark}CC system \eqref{eq43} differs from those in classical mean-field game literature (e.g., \cite{HCM07, LZ08, WZ13, MB17}) by noting the evolution dynamics of $\hat{x}$ is not explicitly specified here. Instead, it is characterized by some implicit representation $\hat{x}=\mathbb{E}_{{\mathcal F}^0}[x_i]$ which is embedded into an augmented mean-field type FBSDE system of $(x_i, k_i, \hat{p}, y, h)$ driven by a generic Brownian motion $W_i$ independent of common noise $W_0.$ CC system \eqref{eq43} is symmetric for all agents thus such representation is uniquely defined.
	
	Such difference in CC representation is mainly caused by the presence of adjoint process term $D^{T} \zeta_i$ in the decentralized strategy design (refer to (\refeq{DD})). Thus, an explicit representation of $\hat{x}$ becomes unavailable, and a similar CC representation was derived in \cite{HHN17}. \end{remark}

\section{Well-posedness of relevant FBSDEs}
Our study in previous sections, especially the one related to decentralized strategy design and consistency condition systems, involves various (fully-coupled) FBSDEs or Riccati equations. Keeping this in mind, this section aims to discuss the existence and uniqueness of their (global) solvability. Note that because of the introduction of soft constraints, these equations are intrinsically \emph{non-standard} (i.e., control/state weights are indefinite) thus their global solvability becomes more technical. Moreover, due to the uncertainty on volatility, it is necessary to treat the adjoint states of FBSDE which closely connect to volatility, the diffusion term in BSDE formulation. As a sequel, the relevant analysis becomes more complex considering the adjoint states are of less regularity property.

We consider the solvability of FBSDE (\ref{eq22}) from Section \ref{sec5}. Similar analysis can be applied to CC system \eqref{eq43} for which the arguments become more lengthy. By partial coupling of Remark \ref{rem5.1}, it suffices to consider the following sub-system constructed by $(h, y,z)$:
\begin{equation}\label{eq42}
\left\{
\begin{aligned}
dh = &Ahdt + \mathbb{I}^G \left(C_0h+R_0^{-1}(z+\hat{\beta}_0)\right)dW_0(t), \quad \\
dy = &-\left(A^Ty+C_0^Tz+(Q-\Xi_1)h+(Q-\Xi_1)\hat{x}-\Xi_2\right)dt\\
&+\left(A^T(G-\Xi_1^G)h+(G-\Xi_1^G)Ah+C_0^T(G-\Xi_1^G)f_2\right)dt+zdW_0(t),\\
h(0)&=0,\quad y(T) =\  (G-\Xi_1^G)\hat{x}(T)-\Xi_2^G.
\end{aligned}\right.
\end{equation}Equation (\ref{eq42}) is a fully coupled FBSDE involving forward state $h$, backward state $y$ and adjoint state $z$. Moreover, it is \emph{nonstandard} or \emph{indefinite} because of the volatility uncertainty (thus, unlike \emph{definite} case, some weights are singular or negative due to its minmax feature). It is known that (global) solvability of such indefinite FBSDE is by no means unconditional: to ensure its well-posedness, it is always necessary to impose some additional compatibility conditions. Also, direction computation indicates the \emph{monotonicity method}, which is well applied to nonstandard FBSDE, fails to work here.

{\textbf{Reduction decoupling method.}} Our method is the reduction decoupling method proposed in \cite{Y06}, which leads to the global solvability by imposing conditions on orthogonality of $C_0$. 
Let $\Psi_1(\cdot,s)$ be the solution of the following ODE:
$$\left\{
\begin{aligned}
\frac{d}{dt}\Psi_1(t,s)&=\begin{pmatrix} A  &0 \\ \widehat{A}&-A^T \end{pmatrix}\Psi_1(t,s), \quad t\in [s,T],\\
\Psi_1(s,s)&=I,
\end{aligned}
\right.$$
where
$$\widehat{A}\stackrel{\Delta}{=}(\Xi_1-Q)+A^T(G-\Xi_1^G)+(G-\Xi_1^G)A+C_0^T(G-\Xi_1^G)
\mathbb{I}^GC_0.$$
Denote $\Psi_1(t)=\Psi_1(t,0)$.
Then we have
\begin{equation*}
\begin{aligned}
\Psi_1(t) = \exp \begin{pmatrix} At  &0 \\ \widehat{A}t&-A^Tt \end{pmatrix}
= \begin{pmatrix} \exp(At)&0 \\  \sum_{n=0}^{\infty}\frac{\Lambda_n t^n}{n!} &\exp(-A^Tt) \end{pmatrix},
\end{aligned}
\end{equation*}
where
$$\Lambda_n\stackrel{\Delta}{=}\widehat{A}A^{n-1}-A^T\widehat{A}A^{n-2}+\cdots+(-A^T)^{k-1}
\widehat{A}A^{n-k}+\cdots+(-A^T)^{n-1}\widehat{A}.$$
If $A=A^T$ and $A\widehat{A}=\widehat{A}A$, then
$$\Lambda_n=\left\{
\begin{array}{cl}
\widehat{A}A^{n-1},& n=2k-1,\\
0,&n=2k,
\end{array}
\right.$$
and
$$\sum_{n=1}^{\infty}\frac{\Lambda_n t^n}{n!} =\widehat{A}\sum_{k=1}^{\infty}\frac{A^{2k-2} t^{2k-1}}{(2k-1)!}. $$
Further, if $A$ is invertible, then
$$\sum_{n=1}^{\infty}\frac{\Lambda_n t^n}{n!} =\widehat{A}A^{-1}\frac{e^{At}-e^{-At}}{2}.$$
We have
\begin{equation*}
\begin{aligned}
& (0,I) \begin{pmatrix} \exp(AT)&0 \\ \sum_{n=1}^{\infty}\frac{\Lambda_n T^n}{n!}&\exp(-A^TT) \end{pmatrix}
\begin{pmatrix} \mathbb{I}^GC_0&0 \\ 0&0 \end{pmatrix} \\
=& 	\begin{pmatrix} \sum_{n=1}^{\infty}\frac{\Lambda_n T^n}{n!}&\exp(-A^TT) \end{pmatrix}
\begin{pmatrix} \mathbb{I}^GC_0&0 \\ 0&0 \end{pmatrix}
=	\begin{pmatrix} \sum_{n=1}^{\infty}\frac{\Lambda_n T^n}{n!}\mathbb{I}^GC_0&0 \end{pmatrix}.
\end{aligned}
\end{equation*}
Thus, if and only if $ \sum_{n=1}^{\infty}\frac{\Lambda_n T^n}{n!}\mathbb{I}^GC_0=0$, then
\begin{equation}\label{eq44a}
(0,I) \Psi_1(T) \begin{pmatrix} \mathbb{I}^GC_0&0 \\ 0&0 \end{pmatrix}=0.
\end{equation}
Note
\begin{align*}
&(0,I)\Psi_1(T)\left(\begin{array}{c} 0\\{-C_0^T+C_0^T(G-\Xi_1^G)\mathbb{I}^GR_0^{-1}} \end{array}\right)
=\exp(-A^T)(-C_0^T+C_0^T(G-\Xi_1^G)\mathbb{I}^GR_0^{-1})=0
\end{align*}implies that
$C_0^T[I-(G-\Xi_1^G)\mathbb{I}^GR_0^{-1}]=0,
$
i.e.,
$C_0^T \mathbb{I}^G=0.$
Since $C_0\not=0$ and $\mathbb{I}^G$ is invertible, we have
$$(0,I)\Psi_1(T)\left(\begin{array}{c} 0\\{-C_0^T+C_0^T(G-\Xi_1^G)\mathbb{I}^GR_0^{-1}} \end{array}\right)\not=0.$$
Note that
\begin{align*}
&(0,I)\Psi_1(T)\left(\begin{array}{c} 0\\I \end{array}\right)
=(0,I)\left(\begin{array}{cc} \exp(AT)&0\\ \sum_{n=1}^{\infty}\frac{\Lambda_n T^n}{n!}&\exp(-A^TT) \end{array}\right)
\left(\begin{array}{c} 0\\I \end{array}\right)\\
&=\Big( \sum_{n=1}^{\infty}\frac{\Lambda_n T^n}{n!}\ \ \exp(-A^TT)\Big)\left(\begin{array}{c} 0\\I \end{array}\right)
=\exp(-A^TT).
\end{align*}
We have $(0,I)\Psi_1(T)\left(\begin{array}{c} 0\\I \end{array}\right)$ is invertible, and
\begin{align}\label{eq44b}
(0,I)\Psi_1(T,t)(R_0^{-1}\mathbb{I}^G,I)^T
=\ &(0,I)\left(\begin{array}{cc} \exp[A(T-t)]&0\\  \sum_{n=1}^{\infty}\frac{\Lambda_n (T-t)^n}{n!} &\exp[-A^T(T-t)] \end{array}\right)\left(\begin{array}{c} \mathbb{I}_GR^{-1}_0\\I \end{array}\right)\cr
=\ &\Big(  \sum_{n=1}^{\infty}\frac{\Lambda_n (T-t)^n}{n!} ,\ \exp(-A^T(T-t)\Big)\left(\begin{array}{c} \mathbb{I}_GR_0^{-1}\\I \end{array}\right)\cr
=\ &  \sum_{n=1}^{\infty}\frac{\Lambda_n (T-t)^n}{n!} \mathbb{I}_GR_0^{-1}+\exp[-A^T(T-t)].
\end{align}
From the above analysis and Theorem 3.2 in \cite{Y06}, we have the following sufficient condition for solvability of FBSDE (\ref{eq42}).
\begin{proposition}\label{prop7}
	Let \emph{(H1)} hold. Then 
	\emph{(\ref{eq42})} is solvable if
$\sum_{n=1}^{\infty}\frac{\Lambda_n T^n}{n!}\mathbb{I}^GC_0\!=\!0$ and $(0,I)\Psi_1$ $(T,\cdot)(R_0^{-1}\mathbb{I}^G,I)^T$ is full-rank.
\end{proposition}

\begin{example}	Consider the system \emph{(\ref{eq1})}--\emph{(\ref{eq2})} with parameters
$$A\!=\!\!\left(\begin{array}{cc}
	2&0\\0&0 \end{array}\right),~ C_0\!=\!\!\left(\begin{array}{cc}
	3&1\\0&2\end{array}\right),~Q=\left(\begin{array}{cc}
	1&0\\0&0.4 \end{array}\right),~R_0=\left(\begin{array}{cc}
	0.1&0\\0&2 \end{array}\right),~\Gamma=\left(\begin{array}{cc}
	1&0\\0&0.5 \end{array}\right),~G=0.$$
	We have $\widehat{A}=\left(\begin{array}{cc}
	0&0\\0&-0.1 \end{array}\right)$ and $\Lambda_n=0, n=1,2,\cdots$. From \emph{(\ref{eq44b})}, we have
$$(0,I)\Psi_1(T,\cdot)(R_0^{-1}\mathbb{I}^G,I)^T=\exp[-A^T(T-t)]$$
is of row full-rank. By Proposition \ref{prop7}, FBSDE \emph{(\ref{eq42})} is solvable.
\end{example}

\section{Asymptotic optimality}\label{sec8}Based on results of Section 6, we may assume the off-line system (\refeq{eq22}) and consistency system (\ref{eq43}) are well-posed (we do not specify which concrete conditions leads to it because our analysis below only requires the wellposedness of these FBSDE systems) thus the decentralized control set $\check{u}=(\check{u}_1, \cdots, \check{u}_{N})$ is well-defined through (\ref{eq24}). The main theorem of this section states the asymptotic robust social optimality of decentralized decision $\check{u}.$

\begin{definition}
	A set of control laws $\check{u}=(\check{u}_1,\cdots,\check{u}_N)$ has asymptotic robust social optimality, if
	\begin{equation} \label{soc}\left|\frac{1}{N}J_{\rm soc}^{\rm wo}(\check{u})-\frac{1}{N}\inf_{u_i\in{\mathcal U}_c }
	J_{\rm soc}^{\rm wo}({u})\right|=o(1).
	\end{equation}
\end{definition}\begin{theorem}\label{thm10}
	Assume that \emph{(H1)}, \emph{(H2$^\prime$)} and \emph{(H3)} hold, and \emph{(\refeq{eq22})} and \emph{(\ref{eq43})} admit a unique solution, respectively. Then the set of control laws \emph{(\ref{eq24})} has asymptotic robust social optimality with
	$$\left|\frac{1}{N} J_{\rm soc}^{\rm wo}(\check{u})-\frac{1}{N}\inf_{u_i\in{\mathcal U}_c }
	J_{\rm soc}^{\rm wo}({u}) \right|=O(\frac{1}{\sqrt{N}}).$$
\end{theorem}

\subsection{A quadratic functional representation}To verify asymptotic social optimality (\ref{soc}), it is helpful to construct some quadratic representation of worse-case functional ${J}_{\rm soc}^{\rm wo}(u)$ for $u=(u_{1}, \cdots, u_{N}) \in \mathbb{R}^{r \times N}.$
First, recall the compact notations introduced in Section \ref{sec3}, and denote $\textbf{R}=\diag(R, \cdots,R)$, $\bar{\beta}^i=[(\beta_1^i)^T,\cdots,(\beta_N^i)^T]^T$, $i=0,1,\cdots,N$. Then we can rewrite  state (\ref{eq9}) and cost functional (\ref{eq8d})
as follows:
\begin{equation}\label{comp1}
\left\{\begin{aligned}%
d\textbf{x} =& (\textbf{Ax}+\textbf{Bu}+\textbf{1}\otimes f)dt+\sum_{i=1}^N(\textbf{D}_i\textbf{u}+{\mathbf{\sigma}_i}) dW_i\cr&+(\textbf{C}_0\textbf{x}+\textbf{D}_0\textbf{u}-\frac{1}{N}(\textbf{1}\textbf{1}^T\otimes R_0^{-1})\bar{\beta}^0dW_0,\ \\
d\textbf{p}\!=\!&-(\textbf{A}^T\textbf{p}\!-\!\hat{\textbf{Q}}\textbf{x}\!+\!\hat{\bf{\eta}}\!+\!\textbf{C}_0^T\bar{\beta}^0)dt\!+\!\sum_{i=1}^N\bar{\beta}^i dW_i\!+\!\bar{\beta}^0dW_0,
\cr \textbf{x}(0)=&\textbf{x}_{0},\ \textbf{p}(T)\!=\!-{\hat{\textbf{G}}}\textbf{x}(T)\!+\!\hat{\bf{\eta}}_0,
\end{aligned}\right.
\end{equation}
and
\begin{align}\label{comp2}
{J}_{\rm soc}^{\rm wo}(\textbf{u})
=\frac{1}{2}\mathbb{E}\int_0^{T}\Big(|\textbf{x}|_{\hat{\textbf{Q}}}^2-2\hat{\bf{\eta}}^T\textbf{x}
+|\textbf{u}|^2_{\textbf{R}}-
\frac{1}{N}|\bar{\beta}^0|_{\textbf{1}\textbf{1}^T\otimes R_0^{-1}}^2\Big)dt+\frac{1}{2}\mathbb{E}\big(|\textbf{x}(T)|^2_{\hat{\textbf{G}}}-2\hat{\bf{\eta}}^T_0\textbf{x}(T)\big),
\end{align}
where $\hat{\bf{\eta}}=\textbf{1}\otimes \Xi_2$, $\hat{\bf{\eta}}_0=\textbf{1}\otimes \Xi_2^G$, and $\hat{\textbf{Q}}=(\hat{Q}_{ij})$,
$\hat{\textbf{G}}=(\hat{G}_{ij})$ are given by
(\ref{eq3a}). Recall by Remark \ref{rem1}, we may exchange the usage $\textbf{u}=(u_1^T,\cdots,u_N^T)^T$ with $u=(u_{1}, \cdots, u_{N}).$

Moreover, by the superposition property of linear system (\ref{comp1}), a straightforward calculation implies that for any $(\textbf{u}^{1}, \textbf{u}^{2}; \textbf{x}^{1}_{0}, \textbf{x}^{2}_{0}; \hat{{\eta}}_{0}^{1}, \hat{{\eta}}_{0}^{2})$,\begin{equation*}
\begin{split}
&J_{\rm soc}^{\rm wo}(\textbf{u}^{1}+\textbf{u}^{2}; \textbf{x}^{1}_{0}+\textbf{x}^{2}_{0}; \hat{{\eta}}_{0}^{1}+\hat{{\eta}}_{0}^{2})+J_{\rm soc}^{\rm wo}(\textbf{u}^{1}-\textbf{u}^{2}; \textbf{x}^{1}_{0}-\textbf{x}^{2}_{0}; \hat{{\eta}}_{0}^{1}-\hat{{\eta}}_{0}^{2})\cr=&2\left(J_{\rm soc}^{\rm wo}(\textbf{u}^{1}; \textbf{x}^{1}_{0}; \hat{{\eta}}_{0}^{1})+J_{\rm soc}^{\rm wo}(\textbf{u}^{2}; \textbf{x}^{2}_{0}; {\hat{\eta}}_{0}^{2})\right).
\end{split}
\end{equation*}Thus, $J_{\rm soc}^{\rm wo}$ satisfies the parallelogram law and it is a quadratic functional with respect to control process $\textbf{u}(\cdot)$ 
and initial-terminal condition pair $(\textbf{x}_0, \hat{\eta}_0)$. Then, by the symmetric peoprty of ${J}_{\rm soc}^{\rm wo}(u)$ to inputs $(\textbf{u}(\cdot);\textbf{x}_{0};\hat{\eta}_0)$, the following quadratic representation holds true: \begin{equation}\label{eq59}
\begin{aligned}
J_{\rm{soc}}^{\rm{wo}}(\textbf{u}(\cdot);\textbf{x}_{0};\hat{\eta}_0)=& \langle \textbf{M}_1(\textbf{u}),\textbf{u}\rangle
+ 2\langle \textbf{M}_{12}(\textbf{x}_{0},\hat{\eta}_0),\textbf{u}\rangle
+ \langle \textbf{M}_2(\textbf{x}_{0},\hat{\eta}_0),(\textbf{x}_{0},\hat{\eta}_0)\rangle \cr&+ 2\langle \textbf{M}_{13},\textbf{u}\rangle
+ 2\langle \textbf{M}_{23},(\textbf{x}_{0},\hat{\eta}_0)\rangle + \textbf{M}_{3},
\end{aligned}
\end{equation}for linear bounded self-adoint operators: $
\textbf{M}_1: \mathcal{U}^{\otimes N}_{c} \rightarrow \mathcal{U}^{\otimes N}_{c}, \  \textbf{M}_2:\mathbb{S}^{nN \times nN}, \ \textbf{M}_3 \in \mathbb{R},
$ and linear bounded operators: $
\textbf{M}_{12}:~
\mathbb{R}^{nN} \times \mathbb{R}^{nN} \rightarrow \mathcal{U}^{\otimes N}_{c},\    \textbf{M}_{13} \in \mathcal{U}^{\otimes N}[0,T],\   \textbf{M}_{23}\in
\mathbb{R}^{nN}$ where $\mathcal{U}_{c}^{\otimes N}=
\underbrace{\mathcal{U}_{c} \times  \cdots \times \mathcal{U}_{c}}_{N-\text{fold}},$ where $\langle\cdot \rangle$ denotes the inner product in the sense of $dt\otimes d\mathbb{P}$.   More precisely, we have the following representations.
For operator $\textbf{M}_1$,\begin{equation*}\left\{
\begin{aligned}
&\textbf{M}_1 (\textbf{u})=\textbf{R}\textbf{u}+\textbf{B}^T \textbf{m}_1 + \sum_{i=1}^{N}\textbf{D}_i^T\textbf{n}_1^i + \textbf{D}_0^T\textbf{n}_1^0;\\
&\langle \textbf{M}_1(\textbf{u}),\textbf{u}\rangle=\mathbb{E} \int_{0}^{T} \langle \textbf{R}\textbf{u} + \textbf{B}^T \textbf{m}_1 + \sum_{i=1}^{N}\textbf{D}_i^T\textbf{n}_1^i + \textbf{D}_0^T\textbf{n}_1^0,\textbf{u} \rangle ds,
\end{aligned}\right.
\end{equation*}with\begin{equation}\label{eq60}
\left\{
\begin{aligned}
&d\textbf{m}_1 = -(\textbf{A}^T\textbf{m}_1 + \textbf{C}_0^T\textbf{n}_1^0 + \hat{\textbf{Q}}\textbf{x}_1 + \hat{\textbf{Q}}\textbf{y}_1)dt +\sum_{i=1}^{N}\textbf{n}_1^idW_i +\textbf{n}_1^0dW_0 ,\\
&d\textbf{y}_1 = \textbf{A}\textbf{y}_1dt + (\textbf{C}_0\textbf{y}_1+N\textbf{R} \otimes \textbf{q}_1^0+\textbf{R}{\otimes}\textbf{n}_1^0)dW_0,\\
&\textbf{y}_1(0) = 0, \quad \quad \textbf{m}_1(T)=\hat{\textbf{G}}(\textbf{y}_1(T) +\textbf{x}_1(T)),\end{aligned}\right.
\end{equation}\begin{equation}\label{eq52}
\left\{
\begin{aligned}
&d\textbf{x}_1 = (\textbf{A}\textbf{x}_1\!+\!\textbf{B}\textbf{u})dt \!+\! \sum_{i=1}^N \textbf{D}_i\textbf{u}dW_i \!+\! (\textbf{C}_0\textbf{x}_1\!+\!\textbf{D}_0\textbf{u}\!-\!\frac{1}{N}(\textbf{1}\textbf{1}^T\otimes R_0^{-1})\textbf{q}_1^0dW_0,\\
&d\textbf{p}_1 = -(\textbf{A}^T\textbf{p}_1-\hat{\textbf{Q}}\textbf{x}_1+\textbf{C}_0^T\textbf{q}_{1}^0)dt +\sum_{i=1}^{N}\textbf{q}_1^idW_i+\textbf{q}_1^0dW_0,\\
&\textbf{x}_1(0) =0, \quad \textbf{p}_1(T) =-\hat{\textbf{G}}\textbf{x}_1(T).
\end{aligned}\right.
\end{equation}
For operator $\textbf{M}_{12}$, we have
\begin{equation*}\left\{
\begin{aligned}
&\textbf{M}_{12}(\textbf{x},\hat{\eta}_0) = \textbf{B}^T\textbf{m}_{2}+\textbf{D}_0^T\textbf{n}^{0}_{2}+\sum_{i=1}^{N}\textbf{D}_i^T\textbf{n}_{2}^i;\\
&\langle \textbf{M}_{12}(\textbf{x},\hat{\eta}_0),\textbf{u} \rangle = \mathbb{E} \int_{0}^{T} \langle \textbf{B}^T\textbf{m}_{2}+\textbf{D}_0^T\textbf{n}_{2}^{0}+\sum_{i=1}^{N}\textbf{D}_i^T\textbf{n}_{2}^i,\textbf{u}\rangle dt,
\end{aligned}\right.
\end{equation*}
with\begin{equation*}\left\{
\begin{aligned}
&d\textbf{m}_2 = -(\textbf{A}^T\textbf{m}_2+\textbf{C}_0^T\textbf{n}_2^0+\hat{\textbf{Q}}\textbf{x}_2+\hat{\textbf{Q}}\textbf{y}_2)dt + \textbf{n}_2^0dW_0+ \sum_{i=1}^{N}\textbf{n}_2^idW_i, \quad  \\
&d\textbf{y}_2 = \textbf{A}\textbf{y}_2dt+(\textbf{C}_0\textbf{y}_2+N\textbf{R}{\otimes}\textbf{q}_2^0+\textbf{R}{\otimes}\textbf{n}_2^0)dW_0, \\&\textbf{m}_2(T)= \hat{\textbf{G}}(\textbf{y}_2(T)+\textbf{x}_2(T)),\quad \textbf{y}_2(0) = 0
\end{aligned}\right.
\end{equation*}and\begin{equation*}\left\{
\begin{aligned}
&d\textbf{x}_2 = \textbf{A}\textbf{x}_2dt + (\textbf{C}_0\textbf{x}_2-\textbf{R}{\otimes}\textbf{q}_2^0)dW_0, \quad  \\
&d\textbf{p}_2 = -(\textbf{A}^T\textbf{p}_2-\hat{\textbf{Q}}\textbf{x}_2+\textbf{C}_0^T\textbf{q}_2^0)dt + \sum_{i=1}^{N}\textbf{q}_2^idW_i + \textbf{q}_2^0dW_0,\\&\textbf{x}_2(0) = \textbf{x}_0,\quad \textbf{p}_2(T) = -\hat{\textbf{G}}\textbf{x}_2(T) +\hat{\eta}_0.
\end{aligned}\right.
\end{equation*}For operator $\textbf{M}_{13}$, we have\begin{equation*}\left\{
\begin{aligned}
&\textbf{M}_{13} = \textbf{B}^T\textbf{m}_{13}+\textbf{D}_0^T\textbf{n}_{13}^0+\sum_{i=1}^{N}\textbf{D}_i^T\textbf{n}_{13}^i,\\
&\langle \textbf{M}_{13},\textbf{u} \rangle = \mathbb{E} \int_{0}^{T} \langle \textbf{B}^T\textbf{m}_{13} + \textbf{D}_0^T\textbf{n}_{13}^0 +\sum_{i=1}^{N}\textbf{D}_i^T\textbf{n}_{13}^i,\textbf{u}\rangle dt
\end{aligned}\right.
\end{equation*}where\begin{equation*}\left\{
\begin{aligned}
&d\textbf{m}_{13} = -(\textbf{A}^T\textbf{m}_{13}+\textbf{C}_0^T\textbf{n}_{13}^0+\hat{\textbf{Q}}\textbf{x}_3+ \textbf{Q}\textbf{y}_{13} -2\hat{\eta}_s)ds + \textbf{n}_{13}^0dW_0+ \sum_{i=1}^{N}\textbf{n}_{13}^idW_i, \\
&d\textbf{y}_{13} = \textbf{A}\textbf{y}_{13}dt+(\textbf{C}_0\textbf{y}_{13}+N\textbf{R}{\otimes}\textbf{q}_{13}^0+\textbf{R}{\otimes}\textbf{n}_{13}^0)dW_0,\\
&\textbf{y}_{13}(0) = 0, \quad \textbf{m}_{13}(T)= \hat{\textbf{G}}(\textbf{y}_{13}(T)+\textbf{x}_3(T))-\hat{\eta}_0,
\end{aligned}\right.
\end{equation*}and\begin{equation*}\left\{
\begin{aligned}
&d\textbf{x}_3 = (\textbf{A}\textbf{x}_3+ \textbf{1} \otimes \textbf{f})dt +\sum_{i=1}^{N}\sigma_idW_i + (\textbf{C}_0\textbf{x}_3-\textbf{R}{\otimes}\textbf{q}_3^0)dW_0,\\
&d\textbf{p}_3 = -(\textbf{A}^T\textbf{p}_3-\hat{\textbf{Q}}\textbf{x}_3+\hat{\eta}+\textbf{C}_0^T\textbf{q}_3^0)dt + \sum_{i=1}^{N}\textbf{q}_3^idW_i + \textbf{q}_3^0dW_0,\\
&\textbf{x}_3(0) = \textbf{x}_0, \quad \textbf{p}_3(T) = -\hat{\textbf{G}}\textbf{x}_3(T).
\end{aligned}\right.
\end{equation*}
$\textbf{M}_2,\textbf{M}_{23}$, and $\textbf{M}_3$ can be defined similarly.
With above presentations, the Fr\'{e}chet differential of $J_{\rm{soc}}^{\rm{wo}}$ along the variation $\delta \textbf{u}$ can be represented as\begin{equation*}
\begin{aligned}
\quad& \delta J_{\rm{soc}}^{\rm{wo}}(\textbf{u}, \delta \textbf{u})=2\langle \textbf{M}_1 \textbf{u}+\textbf{M}_{12}(\textbf{x},\hat{\eta}_0)+\textbf{M}_{13},\delta \textbf{u}\rangle.
\end{aligned}
\end{equation*}

\subsection{Asymptotic optimality: four step procedure}Given the quadratic representation of $J_{\rm soc}^{\rm wo}({u}):=J^{(N)}_{\rm soc}({u}, \sigma_{0}^{*}(u))$ by (\ref{eq59}), we can verify the asymptotic robust optimality stated in Theorem \ref{thm10},  through the following steps.

\emph{Step 1}. We first analyze the asymptotic convergence of the realized state system. When each agent $\mathcal{A}_{i}$ applies the open-loop decentralized strategy $\check{u}_{i}$ as$$\check{u}_{i}=-R^{-1}\left(v_{i}+B^T (G-\Xi_1^G)h+D_0^T (G-\Xi_1^G)g_2\right),$$then the corresponding realized state $\check{x}_{i}$ is given by the following fully-coupled FBSDE subsystem together with backward and adjoint states $(\check{p}_{i}, \check{\beta}_{i}^{0}, \{\check{\beta}_{i}^{k}\}_{k=1}^{N})$ :
\begin{equation}\label{eq49}
\left\{
\begin{aligned}d\check{x}_{i}=&\left[A\check{x}_{i}-BR^{-1} \left(v_i +B^T(G-\Xi_1^G)h+D_0^T(G-\Xi_1^G)g_2\right)+ f\right]dt\\
&\!+\!\left[- DR^{-1} \left( v_i +B^T(G-\Xi_1^G)h+D_0^T(G-\Xi_1^G)g_2\right)+ \sigma\right]dW_i\\
&\!+\!\left[C_0\check{x}_{i}\!-\!DR^{-1} \left(v_i\!+\!B^T(G\!-\!\Xi_1^G)h\!+\!D_0^T(G\!-\!\Xi_1^G)g_2\right)\!-\!\sigma_{0}^{*}(\check{u})\right]dW_0,\\
d\check{p}_i=&-(A ^T\check{p}_i+C_0^T\check{\beta}_i^0-Q\check{x}_i+\Xi_1\check{x}^{(N)}+\Xi_2)dt+\check{\beta}_i^0dW_0+\sum_{k=1}^N\check{\beta}_i^kdW_k,\\
\check{x}_i(0)&=x, \quad
\check{p}_i(T)=(-G)\check{x}_i(T)+\Xi_1^G\check{x}^{(N)}(T)+\Xi_2^G
\end{aligned}
\right.
\end{equation}where $\check{x}^{(N)}=\frac{1}{N}\sum_{i=1}^{N}\check{x}_{i}, \sigma_{0}^{*}(\check{u})=\frac{R_0^{-1}}{N}\sum_{k=1}^N\check{\beta}_k^0.$
Moreover, $v_{i}=B^Tk_i+D_0^T\zeta_0+D^T\zeta_i$ is defined through the following CC system for a representative agent $\mathcal{A}_i$:
\begin{equation}\label{eq40}
\left\{ \begin{aligned}
dx_i=\ &\left[Ax_i-BR^{-1}(v_i+B^T(G-\Xi_1^G)h+D_0^T(G-\Xi_1^G)g_2)+f\right]dt\\
&+\left[ C_0x_i\!-\!D_0R^{-1}(v_i\!+\!B^T(G\!-\!\Xi_1^G)h\!+\!D_0^T(G\!-\!\Xi_1^G)g_2)\!-\!R_0^{-1}\hat{\beta}_0 \right]dW_0\\
&-\left[ D R^{-1}(v_i+B^T(G-\Xi_1^G)h+D_0^T(G-\Xi_1^G)g_2)-\sigma \right]dW_i,\\
dk_i=&-\big[A^T k_i+C^T_0\zeta_0+\!Qx_i-\Xi_1 \mathbb{E}_{{\mathcal F}^0}[x_i] -\Xi_2+(Q-\!\Xi_1)h\\
&+\mathcal{K}(G,g_2)-\mathcal{K}(\Xi_1^G,g_2)\big]\!dt+\zeta_0dW_0+\zeta_idW_i,\\
d\hat{p}=&-\big[A^T \hat{p}+C_0^T\hat{\beta}_0-(Q-\Xi_1)\mathbb{E}_{{\mathcal F}^0}[x_i]+\Xi_2\big]dt+\hat{\beta}_0dW_0, \\
dy=&-\big[A^T y+C_0^Tz+(Q-\Xi_1)h+(Q-\Xi_1)\mathbb{E}_{\mathcal{F}^0}[x_i]-\Xi_2\big]dt\\
&+[A^T(G-\Xi_1^G)h+(G-\Xi_1^G)Ah+C_0^T(G-\Xi_1^G)g_2]dt+zdW_0,\\
dh= &\ Ahdt+[I+R_0^{-1}(G-\Xi_1^G)]^{-1}[C_0h+R_0^{-1}(z+\hat{\beta}_0)]dW_0
\end{aligned} \right.
\end{equation}with the initial-terminal condition\begin{equation}\label{eq40ab}
\left\{ \begin{aligned}
& x_i(0)=x, \quad k_i(T)=Gx_i(T)-\Xi_1^G \mathbb{E}_{\mathcal{F}_T^0}[x_i(T)]-\Xi_2^G,\\
& \hat{p}(T)\!=\!(\Xi_1^G\!-\!G)\mathbb{E}_{\mathcal{F}_T^0}[x_i(T)]\!+\!\Xi_2^G,\ y(T)\!=\!(G\!-\!\Xi_1^G)\mathbb{E}_{\mathcal{F}_T^0}[x_i(T)]\!-\!\Xi_2^G,\ h(0)=0.\\
\end{aligned} \right.
\end{equation}Note that all such $N$-subsystems $\left(\check{p}_{j}, \check{\beta}_{j}^{0}, \{\check{\beta}_{j}^{k}\}_{k=1}^{N}\right)_{j=1}^{N}$ of (\ref{eq49}) are further coupled via the worst-volatility $\sigma^{*}_{0}=\sum_{k=1}^N\check{\beta}_k^0$ and they thus frame a fully-coupled and highly-dimensional FBSDE system in $\!L^{2}_{\cal F}(\Omega; C([0,\!T];\mathbb{R}^{nN}\!)) \times  L^{2}_{\cal F}(0, \!T;\mathbb{R}^{nN}\!))   \times L^{2}_{\mathcal{F}}(0, T; \mathbb{R}^{nN^2}\!)$. Regarding system (\ref{eq49})-(\ref{eq40}), we have the following prior estimate:
\begin{proposition}\label{13}
	Let \emph{(H1)}, \emph{(H2$^{\prime}$)} hold. Assume \emph{(\refeq{eq22})} and \emph{(\ref{eq43})} admit a unique solution, respectively. Then\begin{equation}\label{error}\mathbb{E}\sup_{0 \leq t \leq T}\left(|\check{x}^{(N)}-\mathbb{E}_{\mathcal{F}^0}[x_{i}]|^{2}+|\check{p}^{(N)}-\hat{p}|^2\right)+\mathbb{E}\int_{0}^{T}
	|R_{0}\sigma^{*}_{0}(\check{u})-\hat{\beta}_{0}|^{2}dt  \leq c_{0}(\frac{1}{N})\end{equation}for some constant $c_0>0$ independent of $N$ and $i.$ Here, $\check{p}^{(N)}=\frac{1}{N}\sum\check{p}_{i}.$\end{proposition}
\begin{proof} Making the state aggregation of (\ref{eq49}), we have
	\begin{equation}
	\left\{
	\begin{aligned}
	\!d\check{x}^{(N)}=&\left[A\check{x}^{(N)}-BR^{-1} \left(v^{(N)} +B^T(G-\Xi_1^G)h+D_0^T(G-\Xi_1^G)g_2\right)+ f\right]dt\\
	\!	&+\frac{1}{N}\sum_{i=1}^N\left[- DR^{-1} \left( v_i +B^T(G-\Xi_1^G)h+D_0^T(G-\Xi_1^G)g_2\right)+ \sigma\right]dW_i\\
	&\!+\!\left[C_0\check{x}^{(N)}\!-\!DR^{-1} \left(v^{(N)}\!+\!B^T(G\!-\!\Xi_1^G)h\!+\!D_0^T(G\!-\!\Xi_1^G)g_2\right)\!-\!\sigma_{0}^{*}(\check{u})\right]dW_0,\\
	\!\!d\check{p}^{(N)}\!=\!&-\!\big[A ^T\check{p}^{(N)}\!\!+\!C_0^T\check{\beta}_0^{(N)}\!\!+\!(\Xi_1-Q)\check{x}^{(N)}\!\!+\!\Xi_2\big]dt\!+\!\check{\beta}_0^{(N)}dW_0\!+\!\frac{1}{N}\sum_{i=1}^N\sum_{k=1}^N\check{\beta}_i^kdW_k,\\
	\check{x}^{(N)}(0)&=x_0, \quad
	\check{p}^{(N)}(T)=(\Xi_1^G-G)\check{x}^{(N)}(T)+\Xi_2^G,
	\end{aligned}
	\right.
	\end{equation}where $v^{(N)}=\frac{1}{N}\sum_{i=1}^Nv_i$. By (\ref{eq40}), $\mathbb{E}_{\mathcal{F}^0}[x_{i}]$ satisfies
	\begin{equation}\begin{aligned}
	& d\mathbb{E}_{\mathcal{F}^0}[x_{i}]=\left[A\mathbb{E}_{\mathcal{F}^0}[x_{i}]-BR^{-1}(\hat{v}+B^T(G-\Xi_1^G)h+D_0^T(G-\Xi_1^G)g_2)+f\right]dt\\
	&+\big[ C_0\mathbb{E}_{\mathcal{F}^0}[x_{i}]-D_0R^{-1}(\hat{v}+B^T(G-\Xi_1^G)h+D_0^T(G-\Xi_1^G)g_2)-R_0^{-1}\hat{\beta}_0 \big]dW_0, \cr&
	\mathbb{E}_{\mathcal{F}^0}[x_{i}](0)=x_0,
	\end{aligned}\end{equation}
	where $\hat{v}=\mathbb{E}_{\mathcal{F}^0}[B^Tk_i+D_0^T\zeta_0+D^T\zeta_i]$.
Assume (\ref{eq43}) admits a unique solution thus its state component $(k_{i}, \zeta_0, \zeta_i, h)$ should have an upper bound in their $L^{2}-$norms. Thus, $\sup_{0 \leq t \leq T}\mathbb{E}I_{N}^{2}(t)=O(\frac{1}{N})$ with$$I_{N}:=\frac{1}{N}\sum_{i=1}^N\int_0^T\left[- DR^{-1} \left( v_i +B^T(G-\Xi_1^G)h+D_0^T(G-\Xi_1^G)g_2\right)+ \sigma\right]dW_i.$$
Moreover, wellposedness of (\ref{eq43}) implies some compatibility condition holds true and the iterative scheme of coupled FBSDE works. Then, we can apply the standard continuity-dependence estimate between system (\ref{eq49}) and system (\ref{eq40}) to get the estimate (\refeq{error}). \end{proof}

\emph{Step 2}. Given Step 1, we have the estimate to the realized social cost $J_{\rm soc}^{\rm wo}({\check{u}})$.
\begin{proposition}\label{14}  There exists a constant $c_1$ independent of $N$ such that
	$$J_{\rm soc}^{\rm wo}(\check{u})\les N c_1.$$
\end{proposition}

\begin{proof} Consider the following intermediate state:
	\begin{equation*}
	\begin{aligned}
	dx_i= &\left[ Ax_i -BR^{-1} \left( v_i +B^T(G-\Xi_1^G)h+D_0^T(G-\Xi_1^G)g_2\right)+ f\right]dt\\
	&+\left[- DR^{-1} \left( v_i +B^T(G-\Xi_1^G)h+D_0^T(G-\Xi_1^G)g_2\right)+ \sigma\right]dW_i\\
	&+\left[C_0x_i-DR^{-1} \big( v_i +B^T(G-\Xi_1^G)h+D_0^T(G-\Xi_1^G)g_2\right)-R_0^{-1}\hat{\beta}_0\big]dW_0.
	\end{aligned}
	\end{equation*}
	By Proposition \ref{13}, and standard FBSDE estimate, the following estimate holds:
	$$\mathbb{E}\sup_{0 \les t \les T} |\check{x}_{i}(t)-x_{i}(t)|^{2}  \les \frac{c_{1}}{N}.$$
	Then,
	\begin{align*}
	{J}_{\rm soc}^{\rm wo}(\check{u})
	=&\frac{1}{2}\sum_{i=1}^N
	\mathbb{E}\int_0^T\Big\{\big|(x_{i}-\Gamma \hat{x}-\eta)+(\check{x}_i-x_{i})+\Gamma(\hat{x}-\Gamma \check{x}^{(N)})\big|^2_{Q}\cr&\qquad\qquad\qquad+|\check{u}_i|^2_R-|(\sigma_0^*(\check{u})-\hat{\beta}_{0})+\hat{\beta}_{0}|^2_{R_0}\Big\}dt\cr
	&+
	\frac{1}{2}\mathbb{E}|(x_{i}(T)-\Gamma_{0}\hat{x}(T)-\eta_{0})+(\check{x}_i(T)-x_{i}(T))+\Gamma_0(\check{x}^{(N)}(T)-(\check{x}_i(T))|_G^2\cr
	\les& N c_{2} \big(\|f\|_{L^{2}}+\|\sigma\|_{L^{2}}+\|\Xi_{2}\|_{L^{2}}+\|\Xi_{2}^{G}\|_{L^{2}}+O(\frac{1}{N})\big) \les Nc.
	\end{align*}
\end{proof}

\emph{Step 3}. This step aims to address the convexity of $J_{\rm soc}^{\rm wo}({u})$ of (P2). By its quadratic representation (\ref{eq59}), it is equivalent to
$\langle \textbf{M}_1(\textbf{u}),\textbf{u}\rangle \ges 0$. 
Here,
$$\textbf{M}_1 (\textbf{u})=\textbf{R}\textbf{u}+\textbf{B}^T \textbf{m}_1 + \sum_{i=1}^{N}\textbf{D}_i^T\textbf{n}_1^i + \textbf{D}_0^T\textbf{n}_1^0$$
with $(\textbf{m}_1, \textbf{n}_1^i, \textbf{n}_1^0)$ given by (\ref{eq60}). By examining its coupling structure of (\ref{eq60})-(\ref{eq52}), it can be further reformulated via the following problem:
\begin{align*}
{J}_{\rm soc}^{0}(u)
=&\frac{1}{2}\sum_{i=1}^N
\mathbb{E}\int_0^T\Big\{\big|\grave{x}_i
-\Gamma \grave{x}^{(N)}\big|^2_{Q}+|u_i|^2_R-|\grave{\beta}_0^{(N)}|^2_{R_0^{-1}}\Big\}dt,
\end{align*}
where $\grave{\beta}_0^{(N)}=\frac{1}{N}\sum_{i=1}^N\grave{\beta}_i^0 $, and for $i=1, \cdots, N,$
\begin{equation*}\label{eq9b}
\left\{
\begin{aligned}
d\grave{x}_i=&(A\grave{x}_i+Bu_i)dt+Du_idW_i
+\Big(C_0\grave{x}_i+D_0u_i-\frac{R_0^{-1}}{N}\sum_{k=1}^N\grave{\beta}_k^0\Big)dW_0,\\
d\grave{p}_i=&-(A ^T\grave{p}_i+C_0^T\grave{\beta}_i^0-Q\grave{x}_i+\Xi_1\grave{x}^{(N)})dt+\grave{\beta}_i^0dW_0+\sum_{k=1}^N\grave{\beta}_i^kdW_k,\\
\grave{x}_i(0)&=0, \quad
\grave{p}_i(T)=(-G)\grave{x}_i(T)+\Xi_1^G\grave{x}^{(N)}(T).
\end{aligned}
\right.
\end{equation*}
Then $J_{\rm soc}^{\rm wo}({u})$ of (P2) is 
convex if and only if 
${J}_{\rm soc}^{0}(u)\ges 0.$
Notice the upper bound of realized cost functional $J_{\rm soc}^{\rm wo}(\check{u})$ by Proposition \ref{14}, it suffices to consider the perturbation control $\tilde{u}$ satisfying $J_{\rm soc}^{\rm wo}({\tilde{u}})
\les J_{\rm soc}^{\rm wo}(\check{u}) \les Nc_{1}.$ This further implies that
\begin{equation}\|\tilde{u}\|^2_{L^{2}}:=\sum_{i=1}^N\mathbb{E}\int_0^T|\tilde{u}_i(t)|^2dt
\les Nc \end{equation}
by noting (P2) is convex. Also, it implies $\|\delta \tilde{u}\|^2_{L^{2}}:=\sum_{i=1}^N\mathbb{E}\int_0^T|\delta \tilde{u}_i(t)|^2dt\leq Nc_{1}$ with $\delta u_{i}=\check{u}_{i}-\tilde{u}_{i}.$

\emph{Step 4}. This step discusses the Fr\'{e}chet differential of $J_{\rm soc}^{\rm wo}(u).$ Recall the quadratic functional (\ref{eq59}) and notation exchange between $\textbf{u}$ and $u$, we have
\begin{align*}
J_{\rm soc}^{\rm wo}(\textbf{u})=&\langle \textbf{M}_1(\textbf{u}),\textbf{u}\rangle
+ 2\langle \textbf{M}_{12}(\textbf{x}_{0},\hat{\eta}_0),\textbf{u}\rangle
+ \langle \textbf{M}_2(\textbf{x}_{0},\hat{\eta}_0),(\textbf{x}_{0},\hat{\eta}_0)\rangle
\cr&+ 2\langle \textbf{M}_{13},\textbf{u}\rangle
+ 2\langle \textbf{M}_{23},(\textbf{x}_{0},\hat{\eta}_0)\rangle + \textbf{M}_{3}\\
=&\langle \textbf{M}_1(\check{\textbf{u}}),\check{\textbf{u}}\rangle
+ 2\langle \textbf{M}_{12}(\textbf{x}_{0},\hat{\eta}_0),\check{\textbf{u}}\rangle
+ \langle \textbf{M}_2(\textbf{x}_{0},\hat{\eta}_0),(\textbf{x}_{0},\hat{\eta}_0)\rangle\cr&+ 2\langle \textbf{M}_{13},\check{\textbf{u}}\rangle
+ 2\langle \textbf{M}_{23},(\textbf{x}_{0},\hat{\eta}_0)\rangle + \textbf{M}_{3}\quad\quad\quad\quad\quad
(=J_{\rm soc}^{\rm wo}(\check{\textbf{u}}))\\&+\langle\textbf{M}_1( \textbf{u}-\check{\textbf{u}}), \textbf{u}-\check{\textbf{u}}\rangle
+2\langle \textbf{M}_{13},\textbf{u}-\check{\textbf{u}}\rangle\quad\quad  \quad \quad (=J_{\rm soc}^{0}(\textbf{u}-\check{\textbf{u}}))\\&+2\langle \textbf{M}_1 (\textbf{u})+\textbf{M}_{12}(\textbf{x},\hat{\eta}_0)+\textbf{M}_{13}, \textbf{u}-\check{\textbf{u}}\rangle \quad\quad\quad \quad (=\langle \mathcal{D}_{\textbf{u}}J_{\rm soc}^{\rm wo}(\check{\textbf{u}}), \textbf{u}-\check{\textbf{u}}\rangle)\\=&J_{\rm soc}^{\rm wo}(\check{u})+J_{\rm soc}^{0}(u-\check{u})+\sum_{i=1}^{N}\langle \mathcal{D}_{u_{i}}J_{\rm soc}^{\rm wo}(\check{u}), u_{i}-\check{u}_{i}\rangle\end{align*}
where $\mathcal{D}_{u_{i}}J_{\rm soc}^{\rm wo}(\check{u})$ given by (\ref{eq225}) is the componentwise Fr\'{e}chet derivative of $J_{\rm soc}^{\rm wo}$ at $\check{u}$ on $i^{th}$-component coordinate. Moreover, for $u$, by examining the person-by-person optimality procedures in Section \ref{sec4.1}, and duality expression \eqref{eq20d} for auxiliary cost $\hat{J}_{i}$, we have
$$\|\mathcal{D}_{u_{i}}J_{\rm soc}^{\rm wo}(\check{u})-\mathcal{D}_{u_{i}}\hat{J}_{i}(\check{u})\|_{L^{2}}\les \frac{c}{\sqrt{N}}\|\check{u}\|_{L^{2}}$$
for some constant $c$ independent of $N$ and $\check{u}.$

\ms

\begin{proof}(Proof of Theorem \ref{thm10}). Notice that
$$\left|\frac{1}{N} J_{\rm soc}^{\rm wo}(\check{u})-\frac{1}{N}\inf_{u_i\in{\mathcal U}_c }
	J_{\rm soc}^{\rm wo}({u}) \right|=O(\frac{1}{\sqrt{N}})$$
is equivalent to
$$\inf_{u_i\in{\mathcal U}_c }
	J_{\rm soc}^{\rm wo}({u}) \les J_{\rm soc}^{\rm wo}(\check{u}) \les \inf_{u_i\in{\mathcal U}_c }
	J_{\rm soc}^{\rm wo}({u})+O({\sqrt{N}}).$$
The first inequality is trivial. For the second inequality, we need only consider the perturbed control $u$ satisfying $J_{\rm soc}^{\rm wo}({u}) \les J_{\rm soc}^{\rm wo}(\check{u})$ which is bounded in its $L^{2}$-norm by Step 2, namely $||u||_{L^{2}}^{2} \les cN$ with $c$ independent of $N.$ Now, by Steps 3 and 4, for all such perturbed $u$,
\begin{equation}
	\label{eq59b}
	\begin{aligned}J_{\rm soc}^{\rm wo}({u})-J_{\rm soc}^{\rm wo}(\check{u})=&J_{\rm soc}^{0}(u-\check{u})+\sum_{i=1}^{N}\langle \mathcal{D}_{u_{i}}J_{\rm soc}^{\rm wo}(\check{u}), u_{i}-\check{u}_{i}\rangle \cr\ges& \gamma ||\delta u||^{2}_{L^{2}}+\sum_{i=1}^{N}\langle \mathcal{D}_{u_{i}}J_{\rm soc}^{\rm wo}(\check{u}), \delta u_{i}\rangle. \end{aligned}
	\end{equation}Moreover, by Cauchy-Schwarz inequality,
	\begin{equation}
	\label{eq53}
	\begin{aligned}
	&\sum_{i=1}^{N}\langle \mathcal{D}_{u_{i}}J_{\rm soc}^{\rm wo}(\check{u}), \delta u_{i}\rangle
	\les \sqrt{\sum_{i=1}^{N}\|\mathcal{D}_{u_{i}}J_{\rm soc}^{\rm wo}(\check{u})\|^2_{L^{2}}\sum_{i=1}^{N} \|\delta{u}_i\|^2_{L_2}} \\
	\les& c\sqrt{\sum_{i=1}^{N}O(\frac{1}{N})\|\check{u}\|_{L^2}^2\sum_{i=1}^{N} \|\delta{u}_i\|^2_{L_2}}=O(\sqrt{N})
	\end{aligned}
	\end{equation}                  
	where the last inequality is due to (\ref{eq25}) and Proposition \ref{13} of Step 1. Also, note that $\mathcal{D}
	_{u_{i}}\hat{J}_{i}(\check{u})=0$ for $i=1,\cdots, N$ due to the person-by-person optimality and Theorem \ref{thm4}. Thus, the asymptotic optimality (\ref{soc}) follows directly by (\ref{eq59b}) and (\ref{eq53}).  \end{proof}

\section{Concluding remarks}

This paper investigated mean field LQG social control with volatility uncertain common noise. 
Based on a two-step-duality technique, we construct an auxiliary optimal control problem. Through solving this problem combined with consistent mean field approximations, we design a set of decentralized strategies and verify their asymptotically social optimality. An interesting and challenging work for further study is to consider the indefinite $R$ or $R_0$; or state variable enters the term driven by $W_{i}.$

\end{document}